\documentclass[12pt, reqno]{amsart}
\usepackage{amsmath, amsthm, amscd, amsfonts, amssymb, graphicx, color}
\usepackage{mathtools}
\usepackage{microtype}
\usepackage{float}
\usepackage[table]{xcolor}
\usepackage{enumitem}
\usepackage{algorithm}
\usepackage{algorithmic}
\usepackage{listings}
\usepackage[bookmarksnumbered, colorlinks, plainpages]{hyperref}
\usepackage{tikz}
\usetikzlibrary{arrows.meta,positioning,calc,patterns}

\newtheorem{theorem}{Theorem}[section]
\newtheorem{lemma}[theorem]{Lemma}
\newtheorem{proposition}[theorem]{Proposition}
\newtheorem{corollary}[theorem]{Corollary}
\theoremstyle{definition}
\newtheorem{definition}[theorem]{Definition}

\theoremstyle{remark}

\numberwithin{equation}{section}
\counterwithin{figure}{section}

\newcommand{\supp}{\mathrm{supp}}
\newcommand{\Spec}{\mathrm{Spec}}

\title{Inverse Geometric Diffraction by A Cone}
\author[G. Bao]{Gang Bao}
\address{School of Mathematical Sciences, Zhejiang University, Hangzhou 310027, Zhejiang, China. }
\email{baog@zju.edu.cn}	
\author[X. Chen]{Xi Chen}
\address{Shanghai Center for Mathematical Sciences, Fudan University, Shanghai 200438, China; School of Mathematical Sciences, Fudan University, Shanghai 200433, China; 
	Center for Applied Mathematics, Fudan University, Shanghai 200433, China. }
\email{xi\_chen@fudan.edu.cn}
\author[S. Lu]{Shuai Lu}
\address{School of Mathematical Sciences, Fudan University, Shanghai 200433, China. }
\email{slu@fudan.edu.cn}
\author[K. Xiong]{Kuangmiao Xiong}
\address{Shanghai Center for Mathematical Sciences, Fudan University, Shanghai 200438, China. }
\email{24110840017@m.fudan.edu.cn}


\begin{document}
\setcounter{page}{1}

{\small }

\centerline{}
\centerline{}

\maketitle
\begin{abstract}
	 Consider the inverse problem of recovering a strictly convex conical obstacle in $\mathbb{R}^3$ from the diffraction coefficients along with  {arrival directions (lens data) or arrival times} of diffracted waves. The incident wave is a spherical pulse emanating from a point, and the measurements of diffracted waves are taken at an arbitrarily sized receiver placed within the reflection shadow. Specifically, the  {{lens data or arrival times}} determine the location of the tip, whereas the diffraction coefficients  reconstruct the shape of the cone. Since diffraction coefficients are described by half waves over the complement of the cone base in $\mathbb{S}^2$, we reduce inverse diffraction by a cone in $\mathbb{R}^3$ to identifying the reflected wavefront in $\mathbb{S}^2$ and recovering the obstacle using reflected rays on the sphere. The former is accomplished by constructing the Hadamard parametrix for half waves near the wavefront, whereas the latter relies on the topological properties of broken geodesics on $\mathbb{S}^2$.  
	 The framework developed in this paper exploits the analytic and geometric structures of diffracted wave fields characterized in the Geometrical Theory of Diffraction,   and establishes, for the first time, a rigorous inverse theory corresponding to GTD.
\end{abstract}

%\tableofcontents

\section{Introduction}

\subsection{Overview}

Inverse obstacle problems under limited data are practically valuable yet mathematically challenging. In real world applications, setting up sufficiently many receivers is not only expensive but also  {often} infeasible.  In particular, the reflected rays from  the target obstacle, governed by geometric optics, will be missed by the receivers, when the receivers are located in the reflection shadow. 

While detecting obstacles with some sharp features in the surface, e.g. conic points and edges,  the diffracted waves are illuminated by such singular geometries when incident waves strike there. By Keller's Geometrical Theory of Diffraction (GTD),  the diffraction signals by a corner in $\mathbb{R}^2$ or a cone tip in $\mathbb{R}^3$ propagate omnidirectionally, and thus  {may} be received as long as the diffraction illumination point is visible with respect to the receiver. 

Such diffraction phenomena help obstacle detections, and make it possible to achieve unique inversion with very limited data, even if the receiver is placed in the reflection shadow. In \cite{BCLX1}, we formulated a novel inverse scattering method for  polygons in $\mathbb{R}^2$ by diffracted waves. In this work, we develop the framework of inverse diffraction by  conical obstacles in $\mathbb{R}^3$. The measurements for detection consist of  {arrival directions (lens data) or arrival times} of diffracted waves and diffraction coefficients. In many engineering scenarios, these are locally measurable physical parameters.

The diffraction coefficients of a target obstacle describe how the incident fields diffract by its singular shape, and encode its geometric features and material properties. Specifically, the diffraction coefficients of an obstacle with respect to the spherical incident wave are given as follows.
\begin{itemize}

	\item	Denote by $U^{\bullet}$ and $u^{\bullet}$ the wave fields in the high frequency domain and time domain GTD respectively. In the GTD for both regimes, the full wave fields can be decomposed respectively as
	\begin{align*}
		U^{\mathrm{full}}&=U^{\mathrm{inc}}+U^{\mathrm{ref}}+U^{\mathrm{diff}},	
		\\	
		u^{\mathrm{full}}&=u^{\mathrm{inc}}+u^{\mathrm{ref}}+u^{\mathrm{diff}},
	\end{align*}
	where ${\cdot}^{\mathrm{inc}}$ is the incident wave, ${\cdot}^{\mathrm{ref}}$ is the reflected wave by the smooth obstacle boundary, and ${\cdot}^{\mathrm{diff}}$ is the diffracted wave excited by the singular geometry of the obstacle. %In the point-source setting, $u^{\mathrm{inc}}$ is the free-space Green function.

	\item 
	Let $\widetilde{O} \in \mathbb{R}^3$ be the source point, $P = 0 \in \mathbb{R}^3$ the diffraction point, $\widetilde x := r\omega_{out} \in \mathbb{R}^3$ the observation point, $r > 0$  the distance from $P$ to $\widetilde x$.

	\item In the high frequency regime, the diffracted field $U^{\mathrm{diff}}$ of frequency $k$ excited at $P$ admits an asymptotic expansion of the form   
	$$U^{\mathrm{diff}}(k, \widetilde{x}) = D(k,\widetilde{x}, \widetilde{O}) 2\pi  \frac{\exp(\imath k r)}{kr} U^{\mathrm{inc}}(k,P).$$  The function $D$ is called the frequency-domain diffraction coefficient of the target object. 
	
	\item In the time domain, the diffracted field at observation time $t$ obeys
	\[
	u^{\mathrm{diff}}(t,\widetilde x)
	=
	\frac{1}{r}
	\left(
	d(\cdot- r, \widetilde x, \widetilde{O})
	*
	u^{\mathrm{inc}}(\cdot,P)
	\right)(t).
	\]
	%where   $\ell$ is the total travel time of the incident wave from $O$ to $P$ and the resulting diffracted wave from $P$ to $\widetilde x$. 
	The quantity $d$ is the corresponding time-domain diffraction coefficient.
	
\end{itemize}

%In many engineering scenarios, the lens data / arrival times of diffracted waves and the diffraction coefficients of the target obstacles are locally measurable physical parameters.  It is meaningful to adopt them as local measurements for inverse obstacle. 

In this paper, we establish that given an arbitrarily small observation domain in the reflection shadow and a spherical incident wave from a point, the total travel time from the source point to each observation point and the diffraction coefficient at each point uniquely determine the location and the shape of the target conic obstacle.

\subsection{Inverse geometric diffraction by a cone}
%\subsection{The PDE model}

% $C_P(N)\subset \mathbb{R}^3$. Specifically,  $C_P(N) := [0, \infty) \times N$   has a tip $P = 0 \in \mathbb{R}^3$ and a convex closed domain $N \subset \mathbb{S}^2$ as its base. 

The target obstacle of interest is a strictly convex smooth cone in $\mathbb{R}^3$.
\begin{definition}\label{cone}
	A smooth cone with tip $P$ and base $N\subset \mathbb{S}^2$ is	a domain $C_P(N)\subset \mathbb{R}^3$ satisfying :
	\begin{enumerate}
		\item[(i)] $N$ is closed, and $\partial N$ is smooth;
		\item[(ii)] $\partial C_P(N)\setminus\{P\}$ is the ruled surface
		$$\partial C_P(N)\setminus\{P\} :=
		\left\{
		P+r\omega:\ \omega\in \partial N,\ r>0
		\right\};		$$
		\item[(iii)] $C_P(N)$ consists of
		$$
		C_P(N)
		:=
		\left\{
		P+r\omega:\ \omega\in N,\ r>0
		\right\}\cup\{P\}.
		$$
	\end{enumerate}
	Moreover,  we say that $N\subset \mathbb{S}^2$ is strictly convex if it is a geodesically convex smooth domain contained in an open hemisphere and its boundary has strictly positive geodesic curvature. A smooth cone $C_P(N)$ is  said to be strictly convex if its spherical base $N$ is strictly convex.
\end{definition}

The detection task is conducted in  the corresponding  exterior domain
$
\Omega:=\mathbb{R}^3\setminus C_P(N),
$ which is an open cone with tip $P$ and exterior base
$
M:=\mathbb{S}^2\setminus N.
$
Assume that one can make a disturbance $\delta_{\widetilde{O}}$ at an arbitrarily chosen source point $\widetilde{O} \in \Omega$ and then measure the response signal in an arbitrarily small surface patch $S \subset \Omega$. In particular, we are interested in the case that $S$ is placed in the reflection shadow. 

Mathematically, the inversion consists of spotting the location of the tip $P$ and recovering the shape of $C_P(N)$ (i.e. the shape of the base $N \subset \mathbb{S}^2$). However,  there must exist the shadow of the incident waves from the only source point. Consequently, the response signals do not encode the complete geometric features of $C_P(N)$ but only the front face with respect to the source point. As such, the maximal region of the surface of $C_P(N)$ one can recover is the visible part depicted in Figure~\ref{fig:visible_part_cone}.
\begin{figure}
	\centering
	\includegraphics[width=0.9\textwidth]{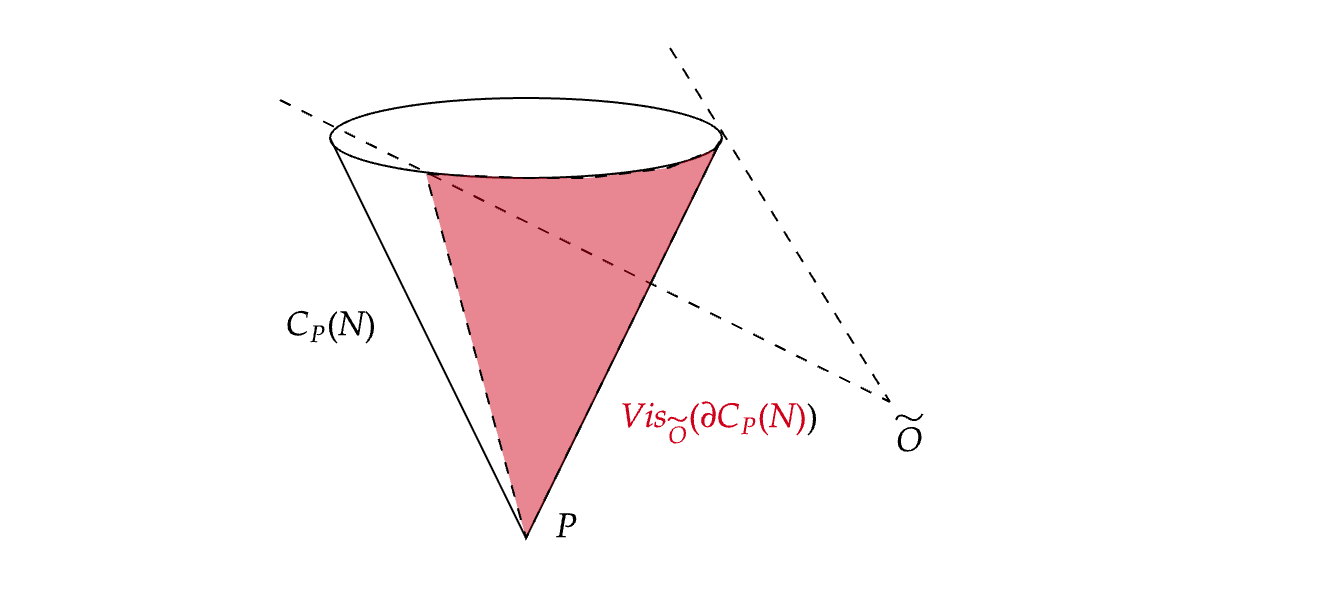}
	\caption{Visible part of $C_P(N)$ with respect to $\widetilde O$.}
	\label{fig:visible_part_cone}
\end{figure}
\begin{definition}\label{visible}
	Let $C_P(N)\subset \mathbb{R}^3$ be a strictly convex cone, $\Omega:=\mathbb{R}^3\setminus C_P(N)$, and  $\widetilde O\in \Omega$. A point $\widetilde x\in \partial C_P(N)$ is said to be  {visible}, with respect to $\widetilde O$, if the line segment joining $\widetilde x$ to $\widetilde O$ is contained in $\Omega\cup\{\widetilde x\}$. The   {visible part} of $\partial C_P(N)$ with respect to $\widetilde O$ is the set
	$
	\operatorname{Vis}_{\widetilde O}\bigl(\partial C_P(N)\bigr)
	$
	of all such visible points.
\end{definition}

The reconstruction of the conical obstacles is conducted by making the active measurements of waves $U^{\bullet}$ in the high frequency domain and waves $u^{\bullet}$ in the time domain. With the positive Laplacian  $\Delta_\Omega$ on $\Omega$, the dynamics of these fields are governed respectively by 
the Dirichlet boundary value problem,
\begin{align*}\label{Helmholtz equation}\left\{\begin{aligned}
		(\Delta_{\Omega} - k^2)U^{\mathrm{full}}(k,\widetilde{x},\widetilde{O})&= \delta_{\widetilde{O}}
		&& \mbox{in $\Omega$},\\
		U^{\mathrm{full}}(k,\widetilde{x},\widetilde{O}) &=0&& \mbox{on $\partial \Omega$},\end{aligned}\right.
\end{align*}  and the Dirichlet initial boundary value problem
\begin{equation*}\label{wave equation}
	\left\{
	\begin{aligned}
		(\partial_{tt}+\Delta_{\Omega})u^{\mathrm{full}}(t,\widetilde{x},\widetilde{O}) &= 0
		&& (t,\widetilde{x})\in \mathbb{R}_+\times \Omega,\\
		u^{\mathrm{full}}(0,\widetilde{x},\widetilde{O}) &= 0
		&& \widetilde{x}\in \Omega, \\
		\partial_t u^{\mathrm{full}}(0,\widetilde{x},\widetilde{O}) &= \delta_{\widetilde O}
		&& \widetilde{x}\in \Omega, \\
		u^{\mathrm{full}}(t,\widetilde{x},\widetilde{O}) &= 0
		&& (t,\widetilde{x})\in \mathbb{R}_+\times \partial\Omega .
	\end{aligned}
	\right.
\end{equation*}

The data for retrieval in the frequency domain are comprised of the lens data of the diffracted wave and the diffraction coefficient at each point   in the receiver :  
\begin{itemize}
	\item The lens data is $(\widetilde{x}, \widetilde{\xi}) \in T \mathbb{R}^3$ where $\widetilde{\xi} \in T_{\widetilde{x}}\mathbb{R}^3$ is the propagating direction of the diffracted field at $\widetilde{x} \in S$;
	\item   The principal frequency domain diffraction coefficient $D_0(\widetilde{x},\widetilde{O})$ is given  specifically by  the following asymptotic expansion of the diffracted field
	\begin{align*}
		U^{\mathrm{diff}}(k,\widetilde{x},\widetilde{O})
		&= \left[U^{\mathrm{inc}} \cdot  2\pi  \frac{\exp(\imath k |\widetilde{O}-P|)}{k|\widetilde{O}-P|} D\right] (k,\widetilde{x},\widetilde{O})  \\
		&=
		2\pi\frac{
			e^{\imath k\left(|\widetilde{O}-P|+|\widetilde{x}-P|\right)}
		}{
			k|\widetilde{O}-P||\widetilde{x}-P|
		}
		\left[D_0(k, \widetilde{x},\widetilde{O})+O(k^{-1})\right].
	\end{align*} 
\end{itemize} 
The  measurements for recovery in the time domain consists of the arrival time of the diffracted field and the diffraction coefficient at each point in the receiver :
\begin{itemize}
	\item The arrival time $\ell$ of the diffracted field at  $\widetilde{x} \in S$  is the distance of the broken line segment from $\widetilde{O}$ via $P$ to $\widetilde{x}$;

	\item The principal time domain diffraction coefficient $d_{0}(t,\widetilde{x},\widetilde{O})$ is given  specifically by  the following asymptotic expansion of the diffracted field 
	\begin{align*}
		u^{\mathrm{diff}}(t,\widetilde{x},\widetilde{O})
		&= \frac{1}{
			|\widetilde{O}-P|
		}(u^{\mathrm{inc}}  \ast d) (t,\widetilde{x},\widetilde{O}) \\
		&=
		\frac{1}{
			|\widetilde{O}-P||\widetilde{x}-P|
		}
		\left[d_0(t- \ell,\widetilde{x},\widetilde{O})+O(t-\ell)\right].
	\end{align*}
\end{itemize}
See Subsection~\ref{subsec:time_domain_diffraction_coefficient} for details on diffraction coefficients.

  { With the notations above, we are ready to present the main result of the paper.}

\begin{theorem}\label{main}
	 Let  $C_P(N) \subset  \mathbb{R}^3$ be a strictly convex smooth cone  with tip $P$ and base $N$ as defined in Definition \ref{cone}.
	Suppose that the source point is  $\widetilde{O}\in \Omega :=  \mathbb{R}^3\setminus C_P(N)$ in the exterior domain, and 
	the receiver is an arbitrarily small $C^2$-surface patch $S\subset \Omega$ such that 
	\begin{itemize}
		\item $S$ lies in the reflection shadow;
		\item $S$ has non-vanishing Gauss curvature;
		
		\item  $S$ is not contained in any hyperboloid.
	\end{itemize}   
 	Then we have the uniqueness of inverse diffraction. More precisely,
 	\begin{itemize}
		\item  	In the high frequency domain, the map	\begin{equation}\label{eqn : FD observation at l}
		\left(P,\operatorname{Vis}_{\widetilde O}\bigl(\partial C_P(N)\bigr)\right)
		\longmapsto
		\left\{
		\left(\widetilde{\xi}(\widetilde{x}),D_{0}(\widetilde{x},\widetilde{O})\right)
		:\ \widetilde{x}\in S
		\right\}
	\end{equation} is injective.
	
\item	In the time domain, the map
	\begin{equation}\label{eqn : observation at l}
		\left(P,\operatorname{Vis}_{\widetilde O}\bigl(\partial C_P(N)\bigr)\right)
		\longmapsto
		\left\{
		\left(\ell(\widetilde{x}),d_0(0,\widetilde{x},\widetilde{O})\right)
		:\ \widetilde{x}\in S
		\right\}
	\end{equation}
	is injective. 	\end{itemize}
\end{theorem}

\begin{figure}
	\centering
	\includegraphics[width=0.8\textwidth]{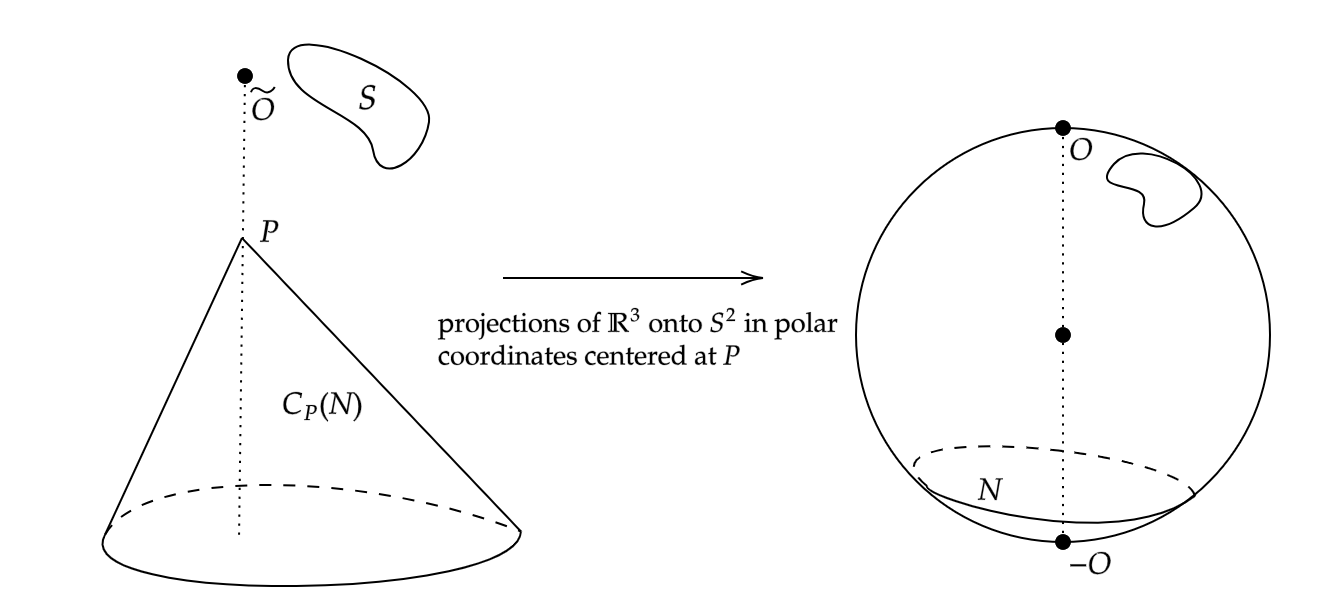}
	\caption{Active measurements  in $\mathbb{R}^3$ and projection to $\mathbb{S}^2$. }
	\label{fig:smooth_cone}
\end{figure}

 \subsection{Geometric theory of diffraction}

The forward formulation for wave diffraction is known as the Geometrical Theory of Diffraction (GTD). The GTD lays out the foundational theory of wave propagation, which classical Geometrical Optics fail to explain.

 The earliest explicit model is Sommerfeld's solution for diffraction by wedges and half-planes \cite{So96,So54} in the frequency domain. Keller and Blank \cite{KellerBlank51} described the evolution of incident, reflected, and diffracted wavefronts generated by plane pulses impinging on wedges and corners, and derived explicit elementary-function solutions. Friedlander \cite{Fr58} studied the diffraction of spherical wave by wedges in the time domain.
 
 Keller's GTD \cite{Ke62} builds on the diffraction law, which relates incident and diffracted rays through a geometrical phase-matching condition, and extends geometrical optics by incorporating diffracted rays emitted from tips, edges, vertices, shadow boundaries, caustics, and other points of geometric singularity. Seckler and Keller \cite{SecklerKeller59} further extended this ray picture to inhomogeneous media. See \cite{KaminetzkyKeller72} for the generalizations to higher-order edges and vertices.
 
 %Ordinary, reflected, refracted, surface, diffracted, and complex rays were formulated through Fermat-type principles and ray tracing. The innovation of GTD is that the asymptotic field is described not only by optical rays, but by an enlarged family of geometrically generated rays. This is the basic mechanism by which wave energy enters shadow regions.

  In the original GTD formulation, the asymptotic ansatz fails near reflection boundaries, where reflected rays emerge or vanish. Consequently, the method yields no physically meaningful field representation within the reflection shadow region. Multiple remedies have been proposed across engineering literature to resolve this deficiency.
  Developed by Kouyoumjian and Pathak \cite{KouyoumjianPathak74}, the Uniform Theory of Diffraction (UTD) circumvents the above restriction via uniformly valid diffraction coefficients that retain correctness throughout such transition zones. See \cite{RousseauPathak95} for the time-domain counterpart of UTD. Alternatively, Ufimtsev’s Physical Theory of Diffraction (PTD) \cite{Ufimtsev71,Ufimtsev96} extends the GTD into the reflection shadow by incorporating supplementary diffractive correction terms.

 %PTD decomposes induced surface currents into a physical-optics part and a nonuniform edge current. The diffracted field is interpreted as radiation from this edge current. While GTD emphasizes rays and wavefront geometry, PTD emphasizes currents and edge radiation. Together they show that edge diffraction is not a secondary correction but a leading mechanism in high-frequency scattering, especially in radar cross-section analysis and stealth design.

 From the viewpoint of microlocal analysis, H\"ormander and Duistermaat's \cite{Ho_FIO_1,Ho_FIO_2} classical propagation of singularities well explains geometric optics in smooth media; however,  {this theory doesn't apply to diffraction of waves}. Cheeger and Taylor \cite{CT82a,CT82b} computed the full wave kernel on conical manifolds by the method of separation of variables, and established the propagation of  singularities for diffracted fields. See \cite{Leb97,MW04,GW-JDG-22,MVW08,V08,dHUV-2015,hintz2024localtheorywaveequations} for microlocal results on propagation of singularities for diffracted fields  in more complex geometry. 
 
 Numerically, Smyshlyaev \cite{Smyshlyaev90} developed the formula for computing diffraction coefficients by cones. The leading term of $D\!\bigl(k,\widetilde{x},\widetilde{O}\bigr)$ is the $-\pi$-time spherical half wave on the exterior base $M$; more precisely, by \cite[(3.5)]{Smyshlyaev90},
 \begin{equation}\label{diffraction coefficient}
 	D_0(\widetilde{x},\widetilde{O})=
 	\Bigl(\exp\Big(-\imath\pi\sqrt{\Delta_M+\tfrac14}\Big)\delta_O\Bigr)(x).
 \end{equation}
 where $\Delta_M$ is the positive Laplacian on $M\subset \mathbb{S}^2$, $x=\pi_{\mathbb S^2}^P(\widetilde x)$ and $O=\pi_{\mathbb S^2}^P(\widetilde O)$ are the spherical projections with center $P$. See \cite{BonnerGrahamSmyshlyaev05,Babich94,BabichDementcprimeSamokishSmyshlyaev02,BabichDementcprimeSamokish00,1} for further applications of Smyshlyaev's formula to various types of waves and boundary conditions.  
 
 {Diffraction phenomena evidently complicate wave fields. Nevertheless, diffracted waves can facilitate inversion by reducing the amount of data required for reconstruction. Recently, \cite{BCLX1} proved the unique determination of polygonal obstacles in \(\mathbb{R}^2\) from measurements collected by receivers of arbitrary size placed at arbitrary locations. The argument relies on explicit computations of diffraction on conic manifolds presented in \cite{CT82a, CT82b}.}

 \subsection{Methodologies}\label{sec : methodlogy}
 
 The overall strategy to prove Theorem \ref{main} consists of 
 \begin{itemize}
 	\item[I.]  spotting the tip by the lens data in the frequency domain  {or} the arrival times in the time domain;
 	
 	\item[II.] retrieving the spherical half waves associated with diffracted waves from diffraction coefficients; 
 	
 	\item[III.]  reconstructing the spherical wavefront by the spherical half wave data given by the diffraction coefficients; 
 	
 	\item[IV.] recovering the visible part of $\partial M$ by the knowledge of the spherical wavefront obtained in Step III.
 \end{itemize}
 
 First of all, tip spotting can be achieved by elementary geometry, since the tip is the point source of diffracted waves. See Proposition \ref{prop:tip_recovery} for details.
 
 Following the Cheeger--Taylor approach on conic manifolds, it is convenient to reduce the shape retrieval in $\Omega \subset \mathbb{R}^3$, via separation of variables, to that in the base $M \subset \mathbb{S}^2$.
 Once the tip is known, the reconstruction of the shape of $C_P(N)$ amounts to the determination of $N$ in $\mathbb{S}^2$.
 
  In view of Smyshlyaev's formula \eqref{diffraction coefficient} and its time-domain realization, the diffraction coefficients in both domains encode the $-\pi$-time spherical half waves. See Proposition \ref{prop:diffraction_to_half_wave}.
 
 Proposition \ref{prop:half_wave_to_wavefront} for Step III aims to retrieve,  by the $-\pi$-time spherical half waves,  the $\pi$-time spherical wavefront
 $$
 \Lambda_M^O:=     \mathrm{sing\,supp} \left(  \exp\left(    \imath \pi \sqrt{\Delta_M + \tfrac{1}{4}}  \right) \delta_O  \right)  \subset M.
 $$
  Proposition \ref{jordan curve theorem} shows that $\Lambda_M^O$ is a simple closed curve such that $M \setminus \Lambda_M^O$ has exactly two connected components $M^{\mathrm{I}}$ and $M^{\mathrm{II}}$. One of the two, say $M^{\mathrm{I}}$, is just the spherical projection of the reflection shadow in $\Omega$, as demonstrated in Proposition \ref{lemma:reflection_shadow_visible_characterization}. In fact, Corollary \ref{WF(u_half)} proves that the half sine wave $$u_{\mathrm{half}}^{\mathrm{sin}}(t, x) := \sin\left(     t \sqrt{\Delta_M + \tfrac{1}{4}}  \right) \delta_O (x)$$ is real-analytic in $\{\pi\} \times  M^{\mathrm{I}}$. Furthermore, Proposition \ref{half wave asymptotic prop}  constructs the Hadamard parametrices for $u_{\mathrm{half}}^{\mathrm{sin}}(\pi, x) $ near $\Lambda^O_M$,  blowing up as $M^{\mathrm{I}} \ni x  \rightarrow \Lambda^O_M$. This identifies $M^{\mathrm I}$ as the maximal domain of real-analytic continuation, with $\Lambda_M^O$ arising as its boundary. Consequently, $\Lambda_M^O$ can be uniquely determined by $u_{\mathrm{half}}^{\mathrm{sin}}(\pi, x) $.

For Step IV, Proposition \ref{reconstruct_boundary}, with the knowledge of $\Lambda_M^O$, reconstructs the visible part of the conic boundary $\partial C_P(N)$. More precisely,  $\Lambda_M^O$ determines the set of spherical reflection points on $\partial M$, which are just  the projection of the visible boundary with respect to $O$. Then the radial extension from $P$ uniquely   recovers the visible part of $\partial C_P(N)$.

In the end, we summarize the inversion process above in Figure~\ref{fig:intro_reconstruction_scheme}.

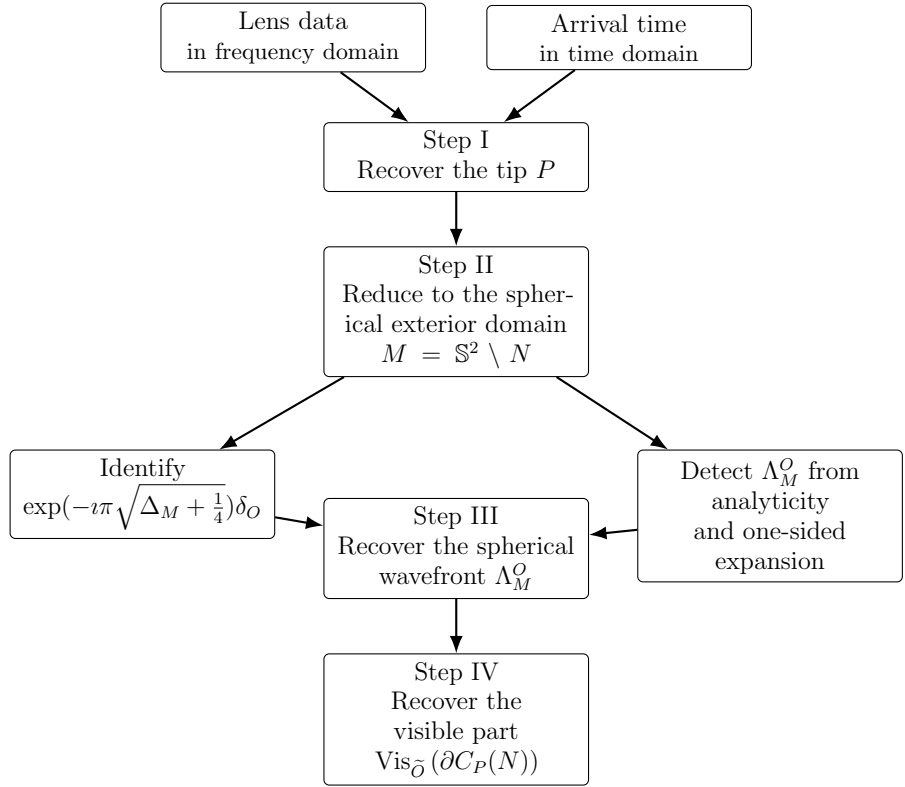
\begin{figure}[htbp]
	\centering
	\begin{tikzpicture}[
		scale=0.8,
		transform shape,
		node distance=9mm and 10mm,
		box/.style={
			rectangle,
			draw,
			rounded corners=2pt,
			align=center,
			text width=4.1cm,
			minimum height=9mm
		},
		arrow/.style={
			thick,
			->,
			>=Latex
		}
		]
		
		\node[box] (lens) {Lens data\\ {\small in frequency domain}};
		\node[box, right=of lens] (data) {Arrival time\\ {\small in time domain}};
		
		\node[box, below=14mm of $(lens)!0.5!(data)$] (tip) {$\text{Step I}$\\Recover the tip \ensuremath{P}};
		\node[box, below=of tip] (sphere) {$\text{Step II}$\\Reduce to the spherical exterior domain\\ \ensuremath{M=\mathbb S^2\setminus N}};
		
		\node[box, below left=12mm and 8mm of sphere] (identify) {Identify\\ \ensuremath{	\exp(-\imath\pi\sqrt{\Delta_M+\tfrac14})\delta_O}};
		\node[box, below right=12mm and 8mm of sphere] (detect) {Detect \ensuremath{\Lambda_M^O} from analyticity\\ and one-sided expansion};
		
		\node[box, below=20mm of sphere] (wavefront) {$\text{Step III}$\\Recover the spherical wavefront \ensuremath{\Lambda_M^O}};
		\node[box, below=of wavefront] (cone) {$\text{Step IV}$\\Recover the visible part\\ $\operatorname{Vis}_{\widetilde O}
			\left(
			\partial C_P(N)
			\right)$};
		
		\draw[arrow] (lens) -- (tip);
		\draw[arrow] (data) -- (tip);
		\draw[arrow] (tip) -- (sphere);
		\draw[arrow] (sphere) -- (identify);
		\draw[arrow] (sphere) -- (detect);
		\draw[arrow] (identify) -- (wavefront);
		\draw[arrow] (detect) -- (wavefront);
		\draw[arrow] (wavefront) -- (cone);
		
	\end{tikzpicture}
	\caption{Reconstruction scheme.}
	\label{fig:intro_reconstruction_scheme}
\end{figure}

\subsection{Structure of the paper}

The proof of Theorem~\ref{main}  adheres to Steps I--IV outlined above. In particular, Steps II--IV are encapsulated into three stand-alone black-box modules for analytical convenience. % in Sections \ref{sec:reduction_to_spherical_problem}--\ref{sec : spher obst} respectively.
First, Section~\ref{sec:reduction_to_spherical_problem} reviews diffracted wavefronts and diffraction coefficients by a cone, and then reduces the diffraction data to the relevant spherical half waves. Prior to conducting further spherical analysis for the inversion scheme, we temporarily step away from the reconstruction framework and derive a set of useful topological properties of spherical wavefronts, which are detailed in Section \ref{sec : spher geom}. Then Section~\ref{sec:determination_of_spherical_wavefronts} establishes the unique recovery of spherical wavefronts from the knowledge of spherical half waves. Subsequently, the spherical obstacle is reconstructed from spherical wavefronts in Section \ref{sec : spher obst}.  The full  proof incorporating Step I and the three modules  is elaborated in  Section~\ref{Reconstruction_of_the_strictly_convex_cone}. For readability, we move the proofs of some stand-alone  technical lemmas to the appendix.

 %Section~\ref{sec:preliminaries_spherical_conic} develops the spherical geometry and diffraction coefficients needed for the reconstruction. Subsection~\ref{sec:proof_of_proposition_D} introduces extended spherical geodesics and the $\pi$-time spherical wavefront, and shows how the spherical wavefront together with its inward normal field determines the visible part of the boundary. 

% Subsection~\ref{The_analytic_singular_support_of_wave_equation} proves real-analyticity of the spherical sine half wave propagator on the $M^{\mathrm I}$. Subsection~\ref{Asymptotic_behavior_of_the_half wave_propagator} derives its one-sided asymptotic expansion along $\Lambda_M^O$, and Subsection~\ref{sec:proof_of_proposition_C} combines these results to recover both $\Lambda_M^O$ and its inward normal field from the half wave data.

%Appendix~\ref{appendix:reflected_rays} collects the auxiliary geometric properties of reflected spherical rays used in the main text. Appendix~\ref{appendix1} proves the factorial tail estimate \eqref{eq:tail_decay_microlocal} required in the analyticity argument.

 \section{Spherical reduction of diffraction}
 \label{sec:reduction_to_spherical_problem}
 
We commence from Step II, and first reduces the diffracted waves in $\Omega$ to the spherical half waves  on \(M\), where the maximal angular travel time \(\pi\) on $M$ governs the dynamics of diffraction in $\Omega$. The main result is   \begin{proposition}
	\label{prop:diffraction_to_half_wave}
	Let $C_P(N_{(j)})$, $j=1,2$, be two cones as in
	Theorem~\ref{main} with the same tip $P$, and set
	$
	M_{(j)}:=\mathbb S^2\setminus N_{(j)}.
	$
	Under the polar coordinates with origin $P$ for $\Omega$, the corresponding spherical projections is denoted by
	\begin{align*}
 \pi_{\mathbb S^2}^{P} : \Omega &\longrightarrow M, \qquad  
   \widetilde{z}  \longmapsto z. 
	\end{align*}
	Suppose that either  the principal frequency-domain diffraction coefficients agree:
	\begin{equation}\label{eqn:frequency-domain diffraction coefficients agree}
		D_{0,(1)}(\widetilde x,\widetilde O)
		=
		D_{0,(2)}(\widetilde x,\widetilde O), \quad \forall \widetilde x\in S,
	\end{equation}
		or  the principal time-domain diffraction coefficients agree at the
		arrival time:
		\begin{equation}\label{eqn:time-domain diffraction coefficients agree}
			d_{0,(1)}\!\bigl(0,\widetilde x,\widetilde O\bigr)
			=
			d_{0,(2)}\!\bigl(0,\widetilde x,\widetilde O\bigr), \quad \forall \widetilde x\in S.
		\end{equation}
	Then the following spherical half waves coincide
	\begin{equation}
		\label{eqn : exp half wave on S}
		\exp\Bigl(-\imath\pi\sqrt{\Delta_{M_{(1)}}+\tfrac14}\Bigr)\delta_O
		=
		\exp\Bigl(-\imath\pi\sqrt{\Delta_{M_{(2)}}+\tfrac14}\Bigr)\delta_O
	\end{equation}
	on $\pi_{\mathbb S^2}^{P}(S)
	\cap
	M_{(1)}^{\mathrm I}
	\cap
	M_{(2)}^{\mathrm I}.$ Consequently, taking the imaginary parts gives
	\begin{equation}
		\label{eqn : sine half wave on S}
		\sin\left(\pi\sqrt{\Delta_{M_{(1)}}+\tfrac14}\right)\delta_O
		=
		\sin\left(\pi\sqrt{\Delta_{M_{(2)}}+\tfrac14}\right)\delta_O
	\end{equation}
	on $\pi_{\mathbb S^2}^{P}(S)
	\cap
	M_{(1)}^{\mathrm I}
	\cap
	M_{(2)}^{\mathrm I}.$
\end{proposition}
 
 \subsection{The diffraction wavefronts in $\Omega$}
 \label{subsec:geometric_pi_front}
 
Consider  $\Omega$ with polar coordinates $(r, x)$ centred at $P$. The standard metric for $\mathbb{R}^3$ induces for $\Omega$   the  metric
 \begin{equation*}\label{eqn : the product metric}
 	g = dr^2 + r^2 g_{\mathbb S^2},
 \end{equation*}
 where $g_{\mathbb S^2}$ is the standard metric on $\mathbb S^2$.
 
The rays associated with scattered waves are the geodesics broken by the obstacle. In particular, a broken geodesic in $\Omega \subset \mathbb R^3$ is a continuous curve which is geodesic in the interior of $\Omega$, and which satisfies the Euclidean specular reflection law whenever it meets $\partial\Omega$.

Let $\gamma$ be the (broken) geodesic segment connecting $(r_1,x_1), (r_2,x_2)\in \Omega$. If $\gamma$ avoids $P$, the length of $\gamma$, by \cite[(3.41)]{CT82a}, is 
\begin{equation*}
	d_{\Omega}((r_1,x_1),(r_2,x_2))=\sqrt{r_1^2+r_2^2+2r_1r_2\cos(\theta(x_1,x_2))},
\end{equation*} 
where $\theta(x_1,x_2)<\pi$ denotes the spherical length on $M$ of the projection of $\gamma$. Thus $\pi$ is the maximal spherical travel time for (broken) spherical geodesics on $M$ induced by trajectories in $\Omega$ that avoid the tip. 
 
 Let
 $
 \widetilde O=(r(\widetilde O),O)\in \Omega,
 $ and $
 \widetilde x=(r(\widetilde x),x)\in \Omega,
 $
 with $O,x\in M$. Denote
 \begin{itemize}
 	\item by $\mathcal G_t$ the union of the incident wavefront and the reflected wavefront;
 	\item by $\mathcal D_t$ the diffracted wavefront;
 	\item by $\mathcal F_t = \overline{\mathcal G_t} \cap \mathcal D_t$ the forward wavefront.
 \end{itemize}  
 For the fixed time $t>0$, the exterior domain $\Omega$ is divided up into three regions I, II and III by the incident wavefront and the diffracted wavefront as in \cite{Fr58,CT82a} (see Figure~\ref{fig:three_regions}).
 
 \begin{figure}[htbp]
 	\centering
 	\includegraphics[width=0.8\textwidth]{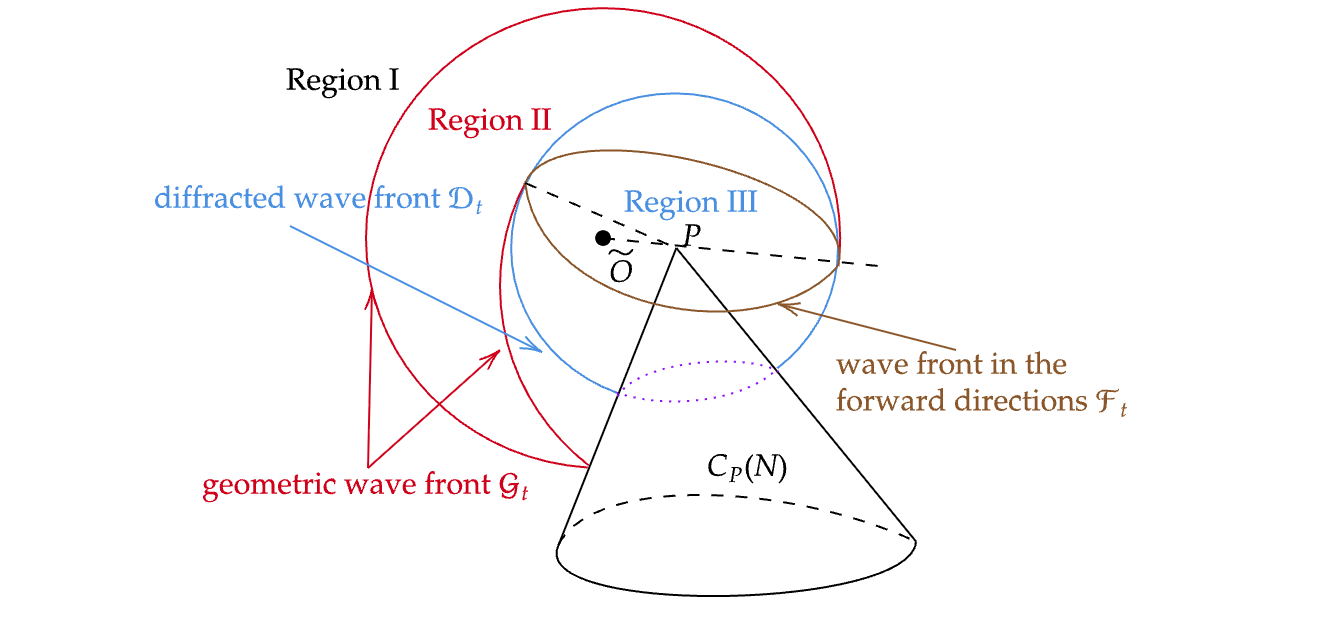}
 	\caption{Diffraction by a cone.}
 	\label{fig:three_regions}
 \end{figure}
 
 More precisely, we have:
 \begin{itemize}
 	\item \textbf{Region I:}
 	\[
 	\left\{\widetilde{x}\in\Omega:t < d_{\Omega}(\widetilde{x},\widetilde{O})\right\}.
 	\]
 	No wave has arrived at $\widetilde x$ at time $t$.
 	
 	\item \textbf{Region II:}
 	\[
 	\left\{\widetilde{x}\in\Omega:d_{\Omega}(\widetilde{x},\widetilde{O}) < t < r(\widetilde x)+r(\widetilde O)\right\}.
 	\]
 	The incident wavefront has passed through $\widetilde{x}$ and the diffracted wavefront has not arrived yet.

 	\item \textbf{Region III:}
 	\[
 	\left\{\widetilde{x}\in\Omega:t > r(\widetilde x)+r(\widetilde O)\right\}.
 	\]
 	All wavefronts have passed through $\widetilde{x}$.
 \end{itemize}
 
 %Accordingly, the $\pi$-time spherical wavefront \eqref{geometric spherical wavefront} is obtained as the spherical projection of the forward wavefront $\mathcal F_t$. Equivalently, $\Lambda_M^O$ characterizes the maximal spherical reach on $M$ generated by geometric wave .

 \subsection{The diffraction coefficients of cones}
 \label{subsec:time_domain_diffraction_coefficient}
 
 Using the Cheeger--Taylor functional calculus on the cone $\Omega=C_P(M)$, we identify the principal diffraction coefficient with a spherical half wave on $M$, where $\Delta_\Omega$ and $\Delta_M$ are the positive Laplacians on $\Omega$ and $M$ respectively.
 
 Let $\{\varphi_i\}_{i\in\mathbb{N}}$ be a collection of eigenfunctions of $\Delta_M$ associated with eigenvalue $\lambda_i$ such that $\{\varphi_i\}_{i\in\mathbb{N}}$ forms an orthonormal basis of $L^2(M)$, and define
 $
 \nu_i := \sqrt{\lambda_i+\tfrac14}.
 $
 By \cite[(0.5)--(0.7)]{CT82a}, for any Borel function $f$, the kernel of $f(\sqrt{\Delta_\Omega})$ admits the expansion
 \begin{equation}\label{eqn:functional_calculus}
 	K_{f(\lambda)}(r_1,x_1,r_2,x_2)
 	=
 	(r_1 r_2)^{-1}
 	\sum_{i=0}^\infty
 	\widetilde K_{f(\lambda)}(r_1,r_2,\nu_i)
 	\varphi_i(x_1)\overline{\varphi_i(x_2)},
 \end{equation}
 where
 \[
 \widetilde K_{f(\lambda)}(r_1,r_2,\nu_i)
 :=
 \int_0^\infty
 f(\lambda)
 J_{\nu_i}(\lambda r_1)
 J_{\nu_i}(\lambda r_2)\lambda\,d\lambda,
 \]
 and $J_\nu(x)$ is the Bessel function of first kind. In particular, by applying \eqref{eqn:functional_calculus} to
 \[
 f_k(\lambda)=(\lambda^2-k^2-\imath0)^{-1},
 \] 
 one can obtain the principal frequency-domain diffraction coefficient \eqref{diffraction coefficient}.  Smyshlyaev obtained the same formula by a direct computation in \cite[(3.5)]{Smyshlyaev90}.

 %Let
 %\[
 %\ell(\widetilde x)=|\widetilde O-P|+|\widetilde x-P|
 %\]
 %be the arrival time of diffracted field. The diffracted field near $t=\ell$ is written as
 %\begin{equation*}\label{eqn:time_field}
% 	u^{\mathrm{diff}}(t,\widetilde x,\widetilde O)
% 	=
% 	\frac{1}{|\widetilde O-P|\,|\widetilde x-P|}
% 	d(t-\ell,\widetilde x,\widetilde O),
% \end{equation*}
% where $d$ is the time-domain diffraction coefficient and $d_0$ denotes its principal term at $t=\ell$, in the sense of $t \to \ell$.
 
Fourier transforming the time-domain diffracted field  $u^{\mathrm{diff}}(t,\widetilde x,\widetilde O)$ gives
 \begin{equation*}\label{eqn:freq_field}
 	\widehat u^{\mathrm{diff}}(k,\widetilde x,\widetilde O)
 	=
 	2\pi
 	\frac{e^{ik\ell}}{k|\widetilde O-P|\,|\widetilde x-P|}
 	D(k,\widetilde x,\widetilde O),
 \end{equation*}
 where $D(k,\widetilde x,\widetilde O)$ has the asymptotic expansion for large $k$
 \begin{equation}\label{eqn:dc_expansion}
 	D(k,\widetilde x,\widetilde O)
 	\sim
 	\sum_{m=0}^\infty k^{-m}D_m(\widetilde x,\widetilde O).
 \end{equation}
 
 The time-domain coefficient is obtained by inverse Fourier transform:
 \begin{equation}\label{eqn:time_freq_relation}
 	d(\tau,\widetilde x,\widetilde O)
 	=
 	\mathcal F^{-1}_{k\to\tau}
 	\bigl[k^{-1}D(k,\widetilde x,\widetilde O)\bigr],
 	\qquad \tau=t-\ell.
 \end{equation}
 Applying \eqref{eqn:time_freq_relation} to the leading term of \eqref{eqn:dc_expansion} yields,
 \begin{equation*}\label{eqn:heaviside=}
 	d_0(\tau,\widetilde x,\widetilde O)
 	=
 	D_0(\widetilde x,\widetilde O)\,H(\tau).
 \end{equation*}
 The Fourier transform of the lower-order terms in \eqref{eqn:dc_expansion} is bounded by $\tau_+ := \tau H(\tau)$ up to constant. That is,
 \begin{equation*}\label{eqn:asymptotic_time}
 	d(\tau,\widetilde x,\widetilde O)
 	=
 	D_0(\widetilde x,\widetilde O)\,H(\tau)
 	+
 	O(\tau_+),\qquad \tau\to0.
 \end{equation*}
 In particular,
 \begin{equation}
 	\label{eqn:arrival_time_d0_half_wave}
 	d_{0}(0,\widetilde{x},\widetilde{O})
 	=
 	D_0(\widetilde x,\widetilde O)
 	=
 	\left(
 	\exp\left(-\imath\pi\sqrt{\Delta_M+\tfrac14}\right)\delta_O
 	\right)(x).
 \end{equation}
 Hence the principal time-domain diffraction coefficient coincides with the $-\pi$-time spherical half wave on $M$.
 
 \subsection{The spherical half wave on $M$}
 
Assuming the tip \(P\) is known, we now formulate the precise reduction from the diffraction coefficients to the spherical half wave at time $-\pi$.

 \begin{proof}[Proof of Proposition \ref{prop:diffraction_to_half_wave}]
 	%Since the tip $P$ is the same, the projection $\pi_{\mathbb S^2}^{P}$ is common to both cones. Thus, for every $\widetilde x\in S$, $x$ and $O$ are fixed.

 	In view of \eqref{eqn:arrival_time_d0_half_wave}, we see that either \eqref{eqn:frequency-domain diffraction coefficients agree} or \eqref{eqn:time-domain diffraction coefficients agree} implies \eqref{eqn : exp half wave on S}. Taking the imaginary parts gives \eqref{eqn : sine half wave on S}.
 \end{proof}

 \section{Spherical geometry with obstacles}\label{sec : spher geom}

Following spherical reduction, subsequent inversion is performed exclusively on $M$. We briefly pause the reconstruction process to derive key topological results for spherical wavefronts in $M$.

%In what follows, we always assume that $N\subset \mathbb{S}^2$ is strictly convex and $M=\mathbb{S}^2\setminus N$. Some technical lemmas are collected in Appendix~\ref{appendix:reflected_rays}.

\subsection{Extended spherical geodesics and  rays}\label{sec:preliminaries_spherical_conic}

%Accordingly, the $\pi$-time spherical wavefront \eqref{geometric spherical wavefront} is obtained as the spherical projection of the forward wavefront $\mathcal F_t$. Equivalently, $\Lambda_M^O$ characterizes the maximal spherical reach on $M$ generated by geometric wave .

The spherical wavefronts are comprised of spherical rays in $M$. Geometrically, they are extended spherical geodesics in $M$.

For $(x_0,\xi_0)\in T\mathbb{S}^2$ with $|\xi_0|_{\mathbb{S}^2}=1$, let $\gamma_{(x_0,\xi_0)}(t)$ denote the unit-speed spherical geodesic emanating from $x_0$, where $t$ is the arclength parameter. Then there holds 
\begin{equation*}
	\left\{
	\begin{aligned}
		\gamma_{(x_0,\xi_0)}(t) &= x(t), && \gamma_{(x_0,\xi_0)}(0)=x_0,\\
		\dot\gamma_{(x_0,\xi_0)}(t) &= \xi(t), && \dot\gamma_{(x_0,\xi_0)}(0)=\xi_0 .
	\end{aligned}
	\right.
\end{equation*}
It is well-known that the geodesics on $\mathbb{S}^2$ are great circles.

Since the obstacle $N$ is a strictly convex subset of $\mathbb{S}^2$, any great circle tangent to \(\partial N\) intersects \(\partial N\) solely at the point of tangency (see Figure~\ref{fig:tangent_great_circle}).
\begin{lemma}\label{great circle intersection}
	Let $C_x \subset \mathbb{S}^2$ be a great circle tangent to
	$\partial N$ at some point $x \in \partial N$. If $S_{P_1}$ and
	$S_{P_2}$ are the two open hemispheres divided by $C_x$, then either
	$N \subset \overline{S}_{P_1}$ or $N \subset \overline{S}_{P_2}$.
	Moreover,
$
		C_x \cap N = \{x\}.
$
\end{lemma}

\begin{figure}[htbp]
	\centering
	\begin{tikzpicture}[scale=0.8,line cap=round,line join=round]
		% sphere
		\draw[black, thick] (0,0) circle (3);
		
		% tangent great circle
		\draw[blue!70!black, thick]
		(-2.9,-0.45) arc[start angle=180,end angle=360,x radius=2.9,y radius=0.52];
		\draw[blue!70!black, dashed, thick]
		(-2.9,-0.45) arc[start angle=180,end angle=0,x radius=2.9,y radius=0.52];
		\node[blue!70!black] at (-1.35,-0.7) {$C_x$};
		
		% obstacle N
		\fill[pattern=north east lines, pattern color=red!75]
		(0,-1)
		.. controls (0.85,-1) and (1.15,-1.20) .. (1.05,-2.00)
		.. controls (0.95,-2.75) and (0.20,-3.00) .. (-0.55,-2.55)
		.. controls (-1.00,-2.20) and (-1.02,-1.15) .. (-0.72,-1.1)
		.. controls (-0.48,-1) and (-0.20,-1) .. (0,-1);
		
		\draw[red!80!black, thick]
		(0,-1)
		.. controls (0.85,-1) and (1.15,-1.20) .. (1.05,-2.00)
		.. controls (0.95,-2.75) and (0.20,-3.00) .. (-0.55,-2.55)
		.. controls (-1.00,-2.20) and (-1.02,-1.15) .. (-0.72,-1.1)
		.. controls (-0.48,-1) and (-0.20,-1) .. (0,-1);
		
		% tangency point
		\fill (0,-1) circle (1.6pt);
		\node[above left] at (0,-1) {$x$};
		
		% labels
		\node at (-1.0,1.15) {$M$};
		\node[red!80!black] at (0.18,-1.8) {$N$};
		\node at (1.85,0.80) {$S_{P_1}$};
		\node at (1.65,-1.55) {$S_{P_2}$};
	\end{tikzpicture}
	\caption{A tangent great circle meets \(N\) only at the tangency point.}
	\label{fig:tangent_great_circle}
\end{figure}
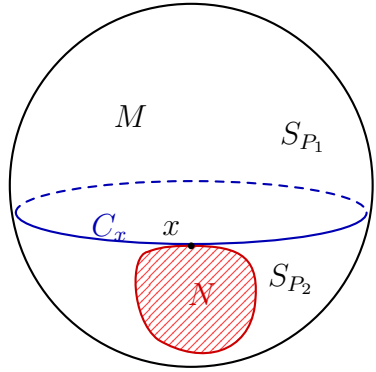

 As shown in Figure~\ref{fig:reflection-grazing-rays}, spherical reflections occur within $M$, and we employ extended spherical geodesics to characterize wave propagation trajectories on $M$.
\begin{figure}[htbp]
	\centering
	\includegraphics[width=0.7\textwidth]{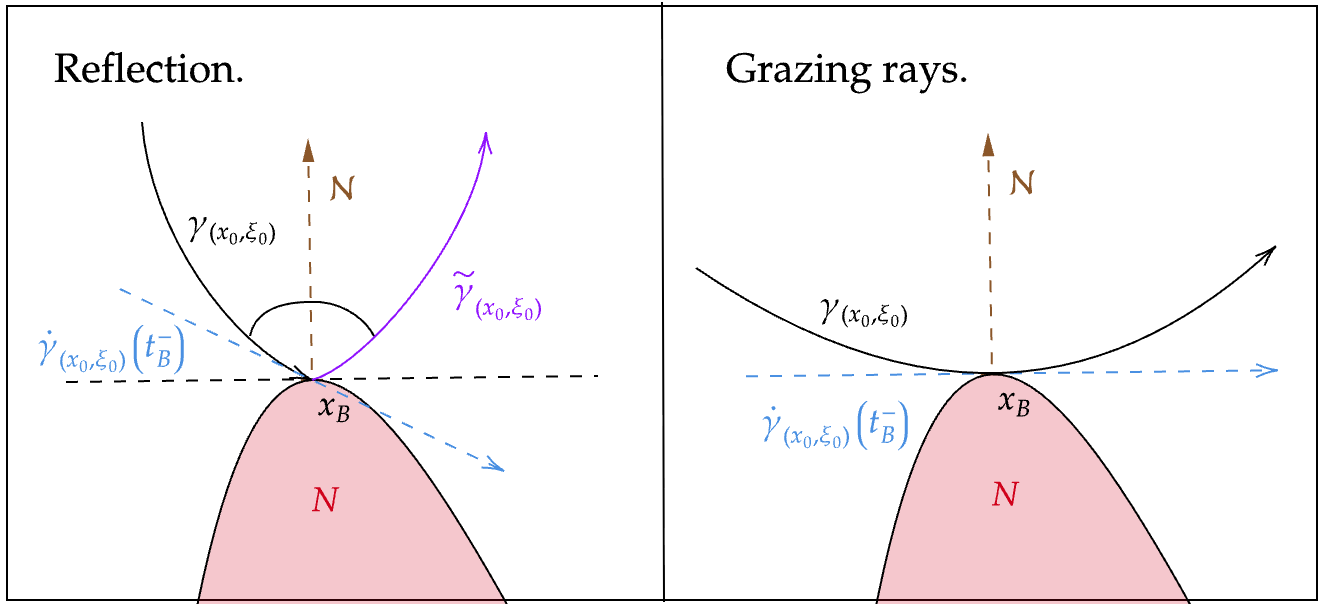}
	\caption{Reflection and grazing rays.}
	\label{fig:reflection-grazing-rays}
\end{figure}

\begin{definition}\label{extended spherical geodesic flow}
The  {extended spherical geodesic}  $\gamma^B_{(x_0,\xi_0)}(t):\mathbb{R}_+\to\overline M$ with initial data   $(x_0,\xi_0)\in SM$ is a  continuous curve, constructed as follows.
	
	\begin{itemize}
		\item 		Before hitting the boundary or if never meeting $\partial M$, $\gamma^B_{(x_0,\xi_0)}$ agrees with $\gamma_{(x_0,\xi_0)}$, i.e.
		$$
		\gamma^B_{(x_0,\xi_0)}(t):=\gamma_{(x_0,\xi_0)}(t), \qquad \mbox{for $t\in[0,t_B]$ or $t \in [0, \infty)$,} 
		$$   where $t_B>0$ is the first hitting time such that
		$ 
		x_B:=\gamma_{(x_0,\xi_0)}(t_B)\in\partial M$ and 
		$\gamma_{(x_0,\xi_0)}([0,t_B))\subset M.
		$
		\item   If reflection occurs at time $t_B$, i.e.
		$$\dot\gamma_{(x_0,\xi_0)}^{B}(t_B - 0) := \lim_{t\uparrow t_B}\dot\gamma^B_{(x_0,\xi_0)}(t)\notin S_{x_B}\partial M,$$
		we extend $\gamma_{(x_0,\xi_0)}^{B}$ by specular reflection. For $t>t_B$, let $\widetilde\gamma_{(x_0,\xi_0)}(t)$ denote the spherical geodesic in $\mathbb{S}^2$ with initial data 
		$\widetilde\gamma_{(x_0,\xi_0)}(t_B) = x_B$ and
		\begin{equation}\label{eqn:reflection vector}
			\dot{\widetilde\gamma}_{(x_0,\xi_0)}(t_B)
			=
			\dot{\gamma}_{(x_0,\xi_0)}^B(t_B-0)
			-2\langle \dot{\gamma}_{(x_0,\xi_0)}^B(t_B-0),\mathcal N(x_B)\rangle_{\mathbb S^2}
			\mathcal N(x_B)\in T_{x_B}\mathbb{S}^2
		\end{equation}
		where $\mathcal N(x_B)$ is the inward unit normal to $M$ along $\partial M$ at $x_B$. We then set $\gamma_{(x_0,\xi_0)}^B(t):=\widetilde\gamma_{(x_0,\xi_0)}(t)$ for $t\ge t_B$ up to the next intersection with $\partial M$.
		
		\item   If $\dot\gamma_{(x_0,\xi_0)}^{B}(t_B-0)\in S_{x_B}\partial M$, we extend $\gamma^B_{(x_0,\xi_0)}$ along the corresponding great circle on $\mathbb{S}^2$. By Lemma~\ref{great circle intersection}, 
		$$
		\gamma_{(x_0,\xi_0)}(\mathbb{R}_+)\cap\partial M=\{x_B\}.
		$$
		
		\item Iterating this construction yields a continuous extended curve defined for all $t\ge 0$.
	\end{itemize}
	For $O\in M$, we define the spherical wavefront at time $\pi$ by
	\begin{equation*}\label{geometric spherical wavefront}
		\Lambda_M^O:=\left\{x\in M: x=\gamma^B_{(O,\xi)}(\pi)\ \text{for some }|\xi|_{\mathbb{S}^2}=1\right\}.
	\end{equation*}

\end{definition}

By the strict convexity of $N$, an extended spherical geodesic meets $\partial M$ at most once for $t\in[0,\pi]$. 

\begin{proposition}\label{counting intersection}
	Let $\gamma^B_{(x_0,\xi_0)}:\mathbb{R}_+\to\overline M$ be an extended spherical geodesic. Then
	$$
	\#\bigl\{t\in[0,\pi]:\gamma^B_{(x_0,\xi_0)}(t)\in\partial M\bigr\}\le 1 .
	$$
\end{proposition}

\begin{proof}
	
\begin{figure}[htbp]
	\centering
	\begin{tikzpicture}[scale=0.8,line cap=round,line join=round]
		% sphere
		\draw[black, thick] (0,0) circle (3);
		
		% tangent great circle through x_B=(0,-1)
		\draw[blue!70!black, thick]
		(-2.9,-0.45) arc[start angle=180,end angle=360,x radius=2.9,y radius=0.55];
		\draw[blue!70!black, dashed, thick]
		(-2.9,-0.45) arc[start angle=180,end angle=0,x radius=2.9,y radius=0.55];
		\node[blue!70!black] at (-1.45,-0.68) {$C_{x_B}$};
		
		% obstacle N tangent at x_B=(0,-1)
		\fill[pattern=north east lines, pattern color=red!75]
		(0,-1)
		.. controls (0.85,-1) and (1.15,-1.20) .. (1.05,-2.00)
		.. controls (0.95,-2.75) and (0.20,-3.00) .. (-0.55,-2.55)
		.. controls (-1.00,-2.20) and (-1.02,-1.15) .. (-0.72,-1.10)
		.. controls (-0.48,-1) and (-0.20,-1) .. (0,-1);
		
		\draw[red!80!black, thick]
		(0,-1)
		.. controls (0.85,-1) and (1.15,-1.20) .. (1.05,-2.00)
		.. controls (0.95,-2.75) and (0.20,-3.00) .. (-0.55,-2.55)
		.. controls (-1.00,-2.20) and (-1.02,-1.15) .. (-0.72,-1.10)
		.. controls (-0.48,-1) and (-0.20,-1) .. (0,-1);
		
		% reflection point
		\fill (0,-1) circle (1.6pt);
		\node[above left] at (0.55,-0.9) {$x_B$};
		
		% antipodal point of x_B
		\fill (0.2,0.1) circle (1.6pt);
		\node[above right] at (-0.3,0.3) {$-x_B$};
		
		% incoming ray
		\draw[violet, very thick, ->]
		(-1.6,2.25) .. controls (-1.55,1.35) and (-1.42,0.00) .. (0,-1);
		
		% reflected geodesic: a half great-circle through x_B and its antipodal point
		\draw[violet, very thick, ->]
		(0,-1) .. controls (3.6,1.45) and (3.55,2.10) .. (0.2,0.1);

		% annotation for the reflected arc
		\node[violet, align=left] at (6.2,-0.45) 
		{\small a half great circle\\\small in $S_{P_2}$, of length $\pi$};
		\draw[violet, ->, thin] (6.05,0.5) -- (2.75,1.45);
		
		% labels for regions
		\node at (-1.05,1.15) {$M$};
		\node[red!80!black] at (0.15,-1.85) {$N$};
		\node at (1.8,1.65) {$S_{P_2}$};
		\node at (1.8,-1.65) {$S_{P_1}$};
		
		% label for the extended spherical geodesic
		\node[violet] at (-2.1,0.75) {$\gamma^B_{(x_0,\xi_0)}$};
	\end{tikzpicture}
	\caption{After reflection, the ray $\gamma^B_{(x_0,\xi_0)}$ stays in $S_{P_2}$ up to time $\pi$, while $N\subset \overline{S}_{P_1}$.}
	\label{fig:one_intersection_pi}
\end{figure}
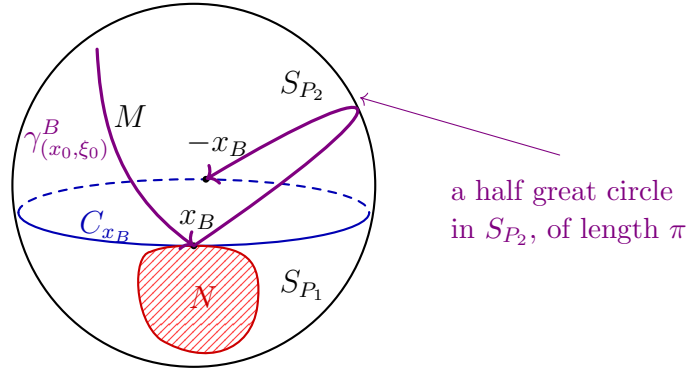

	As shown in Figure~\ref{fig:one_intersection_pi}, an extended spherical geodesic $\gamma^B_{(x_0,\xi_0)}$ first intersects $\partial N$ at $x_B$ at time $t_B$. The great circle $C_{x_B}$ partitions the sphere into two open hemispheres $S_{P_1}$ and $S_{P_2}$. Without loss of generality, assume $N \subset \overline{S}_{P_1}$ by Lemma~\ref{great circle intersection}. By the reflection law, the reflected segment $\gamma^B_{(x_0,\xi_0)}((t_B,t_B+\pi))$ remains in $S_{P_2}$. Since $N\cap S_{P_2}=\emptyset$, there is no second hitting within time \(\pi\).
\end{proof}

Proposition~\ref{counting intersection} motivates the definition of reflected spherical rays.

\begin{definition}\label{the classification of ray}
	Assume $\gamma^B_{(x_0,\xi_0)}(t)\in\partial M$ for some $t\in[0,\pi]$. We denote its first hitting time, first hitting point, and incident direction by
	\begin{equation*}
		\begin{aligned}
			t_B(\xi_0) &:= \min\{t>0:\gamma^B_{(x_0,\xi_0)}(t)\in\partial M\},\\
			x_B(\xi_0) &:= \gamma^B_{(x_0,\xi_0)}(t_B(\xi_0))\in\partial M,\\
			\xi_B(\xi_0)
			&:= \lim_{t\uparrow t_B(\xi_0)} \dot\gamma^{B}_{(x_0,\xi_0)}(t)
			= \dot\gamma_{(x_0,\xi_0)}(t_B(\xi_0))
			\in T_{x_B(\xi_0)}\mathbb S^2 .
		\end{aligned}
	\end{equation*} 
	An extended spherical geodesic $\gamma^B_{(x_0,\xi_0)}(t)$ is called a  {reflected spherical ray} if
	$$
	\#\left\{t\in(0,\pi):\gamma^B_{(x_0,\xi_0)}(t)\in\partial M\right\}=1
	\quad\text{and}\quad
	\xi_B(\xi_0)\notin T_{x_B(\xi_0)}\partial M .
	$$
\end{definition}

The following proposition shows that distinct reflected spherical rays cannot intersect up to time $\pi$.

\begin{proposition}\label{one ray intersection}
	Let $\gamma^B_{(O,\xi_1)}$ and
	$\gamma^B_{(O,\xi_2)}$ be reflected spherical rays on $M$. If there exists $t_0\in(0,\pi]$ such that
	$
	\gamma^B_{(O,\xi_1)}(t_0)=\gamma^B_{(O,\xi_2)}(t_0),
	$
	then $\xi_1=\xi_2$.% Conversely, if $\xi_1\neq \xi_2$, then there holds that	\begin{equation}\label{eqn:second statement}
%		\gamma^B_{(O,\xi_1)}\bigl([t_B(\xi_1),\pi]\bigr)		\cap		\gamma^B_{(O,\xi_2)}\bigl([t_B(\xi_2),\pi]\bigr)=\emptyset.	\end{equation}
\end{proposition}

\begin{proof}
	Let $p$ be the first intersection point, i.e.
	\begin{equation}\label{eqn: equal xi_1 and xi_2}
		\gamma^B_{(O,\xi_1)}(t_0)=\gamma^B_{(O,\xi_2)}(t_0)=p.
	\end{equation}
	We split the proof into three cases, according to whether neither, exactly one, or both rays have reflected before \(t_0\).
	
	\medskip
	\noindent\textbf{Case 1: no reflection before $t_0$.}
	Both of $\gamma^B_{(O,\xi_1)}$ and $\gamma^B_{(O,\xi_2)}$ are minimizing spherical geodesics from $O$ to $p$ of length $t_0<t_B(\xi_1)<\pi$. 
	By uniqueness of minimizing geodesics of length less than $\pi$ on $\mathbb{S}^2$, it follows that $\xi_1=\xi_2$.
	
	\medskip
	\noindent\textbf{Case 2: only one ray reflects.}
	Assume $\gamma^B_{(O,\xi_1)}$ reflects before $t_0$, while $\gamma^B_{(O,\xi_2)}$ does not. Then
	\[
	d_{\mathbb{S}^2}(O,p)=t_0,
	\]
	and $\gamma^B_{(O,\xi_1)}([0,t_0])$ has the same length, hence is also length minimizing. This contradicts the smoothness of length-minimizing geodesics, since  $\gamma^B_{(O,\xi_1)}$ is nonsmooth at $x_B(\xi_1)$ by the reflection law \eqref{eqn:reflection vector}.

	\medskip
	\noindent\textbf{Case 3: both rays reflect.}
	Let $x_1,x_2\in\partial M$ be the first reflection points of
	$\gamma^B_{(O,\xi_1)}$ and $\gamma^B_{(O,\xi_2)}$, respectively, and let
	$C_{x_1}$ be the tangent great circle to $N$ at $x_1$. Assume $x_1\neq x_2$.
	By Lemma~\ref{great circle intersection}, \(C_{x_1}\) meets \(\partial M\) only at \(x_1\) and divides $\mathbb{S}^2$ into $S_{P_1}$ and $S_{P_2}$. Without loss of generality, assume \(N\subset \overline{S}_{P_1}\). Since  
	$
	\gamma^B_{(O,\xi_1)}((t_B(\xi_1),\pi))\subset S_{P_2}$ and $x_2\in N\setminus\{x_1\}\subset S_{P_1},
	$ we see that
	$\gamma^B_{(O,\xi_2)}$ must intersect $C_{x_1}$ again after $x_2$ and before reaching $p$.
	
	Set
	\[
	t_{\mathrm{out}}
	:=
	\sup\{s\in[0,t_0]:
	\gamma^B_{(O,\xi_2)}(s)\in C_{x_1}\},
	\qquad
	x_{\mathrm{out}}
	:=
	\gamma^B_{(O,\xi_2)}(t_{\mathrm{out}}).
	\]
	We now extend $\gamma^B_{(O,\xi_2)}$ beyond $t_{\mathrm{out}}$ by the unique spherical geodesic in $\mathbb S^2$ solving
	\[
	\gamma(t_{\mathrm{out}})=x_{\mathrm{out}},\qquad
	\dot\gamma(t_{\mathrm{out}})
	=
	\dot{\gamma}^B_{(O,\xi_2)}(t_{\mathrm{out}})
	-2\big\langle \dot{\gamma}^B_{(O,\xi_2)}(t_{\mathrm{out}}),
	\nu_{x_1}(x_{\mathrm{out}})\big\rangle_{\mathbb S^2}
	\nu_{x_1}(x_{\mathrm{out}}),
	\]
	where $\nu_{x_1}(x_{\mathrm{out}})$ is the unit normal to $C_{x_1}$ at $x_{\mathrm{out}}$ pointing toward $S_{P_1}$.
	Denote this extension by $\gamma_{(O,\xi_2)}^{B,\mathrm{re}}$.
	
	By construction,   
	\(\gamma_{(O,\xi_2)}^{B,\mathrm{re}}\) for \(t>t_{\mathrm{out}}\) is obtained by specular reflection across \(C_{x_1}\), as illustrated in Figure~\ref{fig:reflected_curve_Cx1}.
	Indeed, in Cartesian coordinates where the plane containing \(C_{x_1}\) is the \(xy\)-plane, this reflection is given by
	\[
	(x,y,z)\mapsto(x,y,-z).
	\]
	Since \(\gamma^B_{(O,\xi_1)}\) is also obtained by reflection across \(C_{x_1}\) at \(x_1\) after $t_B(\xi_1)<t_0$, it follows from \eqref{eqn: equal xi_1 and xi_2} that
	\[
	\gamma_{(O,\xi_2)}^{B,\mathrm{re}}(t_0)=\gamma_{(O,\xi_1)}(t_0),\qquad t_{\mathrm{out}}<t_0\leq \pi.
	\]
	Thus the segment $\gamma_{(O,\xi_2)}^{B,\mathrm{re}}([0,t_0])$ is also length minimizing from $O$ to $\gamma_{(O,\xi_1)}(t_0)$. This gives the same contradiction as in Case 2: the reflection creates a corner on a length-minimizing geodesic.
	
	Therefore, both rays must have the same first reflection point $x_1=x_2$, which implies $\xi_1=\xi_2$.\end{proof}
	
	\begin{figure}
		\centering
		\includegraphics[width=0.8\textwidth]{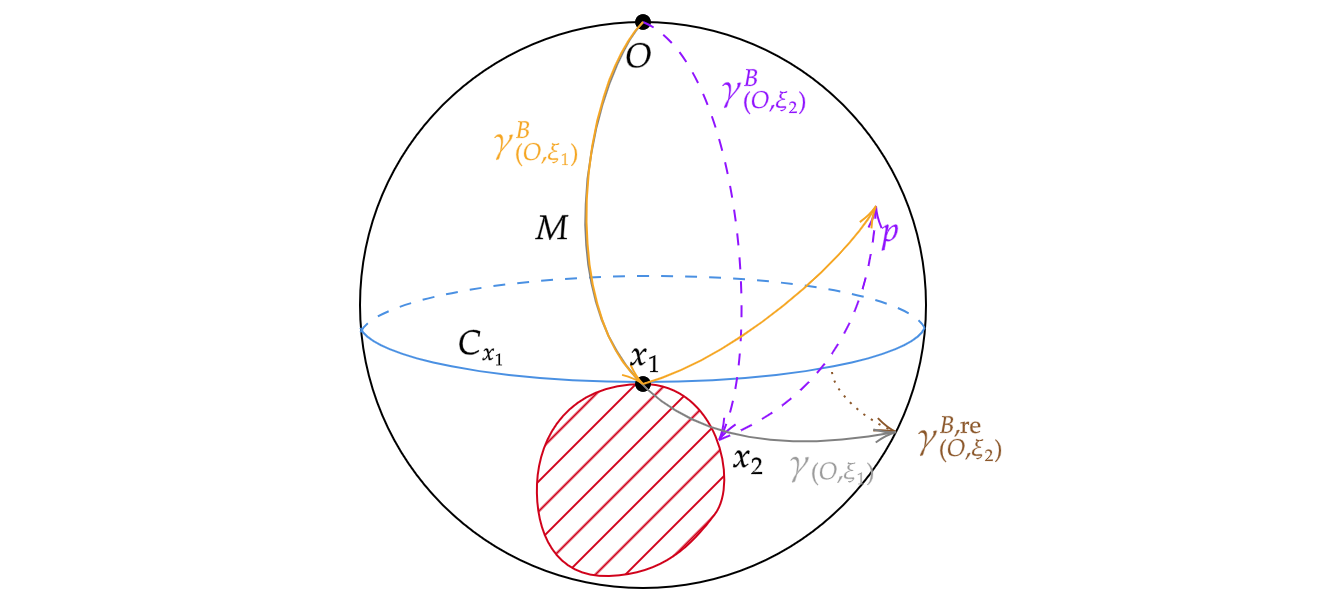}
		\caption{Reflection across the tangent great circle \(C_{x_1}\).}
		\label{fig:reflected_curve_Cx1}
	\end{figure}
	
%	For the second statement, suppose that there exist	\[	t_1\in[t_B(\xi_1),\pi],\qquad t_2\in[t_B(\xi_2),\pi]\]	such that $\gamma^B_{(O,\xi_1)}(t_1)=\gamma^B_{(O,\xi_2)}(t_2).$ Without loss of generality, assume $t_1\ge t_2$. The same argument as in Case 3 gives \(t_1=t_2\) and hence $\xi_1=\xi_2$ by the first statement. This proves \eqref{eqn:second statement}.

\subsection{The spherical wavefront and its topological structure}

We next study the structure of the spherical wavefront $\Lambda_M^O$.

\begin{proposition}\label{stratification}
	If $\partial M$ is smooth, then $\Lambda_M^O\setminus\{-O\}$ is a smooth curve in $\mathbb{S}^2$.
\end{proposition}

\begin{proof}
	If $\gamma_{(O,\eta)}^B$ is not a reflected spherical ray, then $\gamma_{(O,\eta)}^B(\pi)=-O$. Thus it suffices to consider reflected spherical rays.
	
	Let $\gamma_{(O,\eta)}^B$ be a reflected spherical ray. Define the exponential map
	\begin{equation*}\label{eqn : the exponential map}
		\exp_O^{B}:\mathbb{R}_+\times S_OM \to M,\qquad
		\exp_O^{B}(t,\xi):=\gamma^B_{(O,\xi)}(t).
	\end{equation*}
Then there holds that
\begin{lemma}\label{spherical geodesic non-degeneracy}
	For every $t\in \bigl(t_B(\eta),\pi\bigr]$, the differential
	$$
	d_{(s,\xi)}\exp_O^{B}\big|_{(s,\xi)=(t,\eta)}:
	T_{(t,\eta)}\bigl(\mathbb{R}_+\times S_OM\bigr)\to T_{\gamma^B_{(O,\eta)}(t)}M
	$$
	is an isomorphism.
\end{lemma} We defer the proof of Lemma~\ref{spherical geodesic non-degeneracy} to the appendix. Then by the inverse function theorem, there exist a neighbourhood $U$ of $(\pi,\eta)$ in $\mathbb{R}\times S_OM$ and a neighbourhood $V$ of $\exp_O^B(\pi,\eta)$ in $M$ such that
	$
	\exp_O^B:U\to V
	$
	is a diffeomorphism.
	
	In particular,   the map
	$
	\eta'\mapsto \exp_O^B(\pi,\eta')
	$
	provides a local parametrization of $\Lambda_M^O\setminus\{-O\}$ near $\exp_O^B(\pi,\eta)$. Therefore, $\Lambda_M^O\setminus\{-O\}$ is a smooth curve in $\mathbb{S}^2$.
\end{proof}

Next, we show that the spherical wavefront is in fact a simple closed curve. Consequently, the spherical wavefront separates the sphere into the two components $M^{\mathrm I}$ and $M^{\mathrm{II}}$ (see Figure~\ref{fig:jordan_wavefront}).

\begin{figure}
	\centering
	\includegraphics[width=0.8\textwidth]{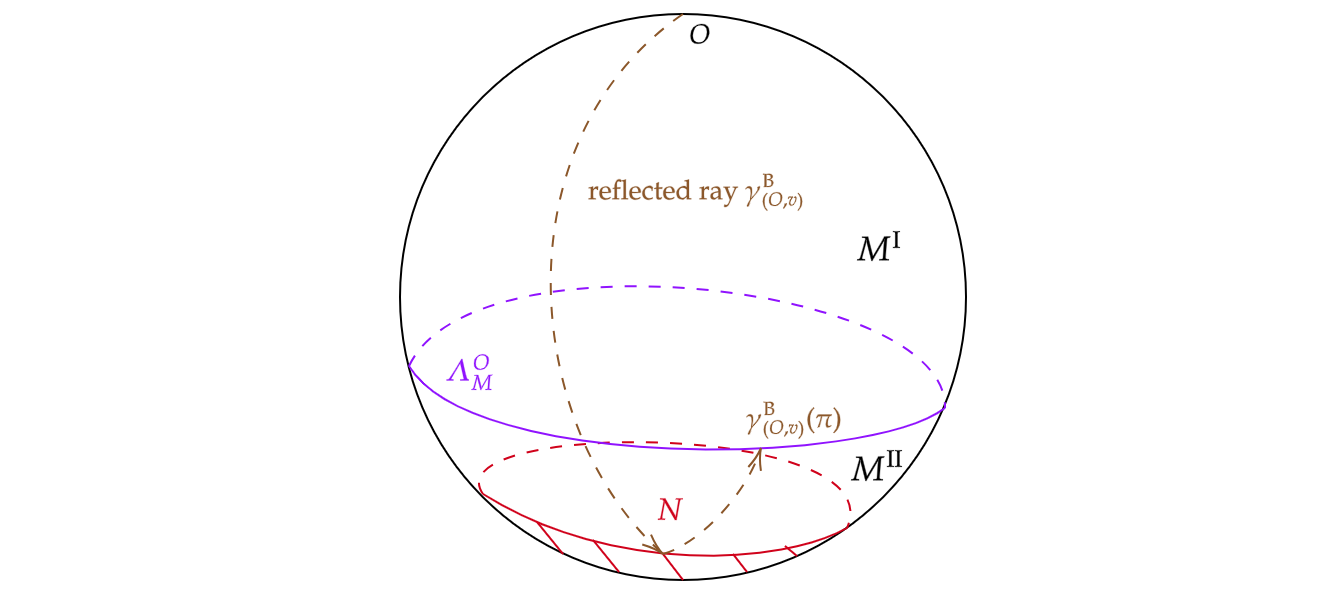}
	\caption{The spherical wavefront is a simple closed curve.}
	\label{fig:jordan_wavefront}
\end{figure}

\begin{proposition}\label{jordan curve theorem}
	The spherical wavefront	$\Lambda_M^O$ is a simple closed curve. Consequently, $\mathbb S^2\setminus \Lambda_M^O$ has exactly two connected components $M^{\mathrm I}$ and $M^{\mathrm{II}}$ such that
	$
	\operatorname{Int}(N)\subset M^{\mathrm{II}}.
	$
\end{proposition}

\begin{proof}
	Set the
	$\pi$-time exponential map
	\begin{equation}\label{eq:pi-time exponential map}
		\exp_O^{B}(\pi,\cdot):\mathbb R/2\pi\mathbb{Z}\cong S_OM\to M,\qquad \xi\mapsto\exp_O^{B}(\pi,\xi).
	\end{equation}
	We distinguish three cases based on the position of $-O$ relative to \(N\).
	
	\medskip
	\noindent\textbf{Case 1: $-O\in \operatorname{Int}(N)$.}
	Every half great circle from $O$ meets $\partial M$ transversely, and whence each extended spherical geodesic reflects exactly once before time $\pi$. Thus \eqref{eq:pi-time exponential map} is injective by Proposition~\ref{one ray intersection}. Since $S_OM\cong\mathbb S^1$ is compact, \eqref{eq:pi-time exponential map} is a topological embedding. Hence $\Lambda_M^O$ is a simple closed curve.
	
	\medskip
	\noindent\textbf{Case 2: $-O\in \partial N$.}
	Let $C_{-O}$ be the great circle tangent to $N$ at $-O$. By Lemma~\ref{great circle intersection}, $C_{-O}\cap N=\{-O\}$. Two unit tangent tangents of $\partial N$ at $-O$ divide $S_OM$ into two open arcs $I_0$ and $I_1$. Extended spherical geodesics with initial direction $\xi_0$ in $I_0$ do not reflect and satisfy $\exp_O^{B}(\pi,\xi_0)=-O$, whereas those in $I_1$ reflect exactly once. Hence $\exp_O^{B}(\pi,\cdot)|_{I_1}$ is injective by Proposition~\ref{one ray intersection}.
	
	If $\xi_\pm\in S_OM$ are the endpoints of $I_1$, continuity of the extended flow gives
	\[
	\lim_{\xi\to\xi_\pm,\,\xi\in I_1}\exp_O^{B}(\pi,\xi)=-O.
	\]
	Since $-O\notin \exp_O^{B}(\pi,\cdot)(I_1)$, the map $\exp_O^{B}(\pi,\cdot)|_{I_1}$ extends by identifying $\xi_-$ and $\xi_+$ to a continuous injective map
	\[
	\mathbb{R}/\left((\xi_+-\xi_-)\mathbb{Z}\right)\cong\mathbb S^1\longrightarrow\mathbb S^2
	\]
	with image $\Lambda_M^O$. This map is an embedding, and whence $\Lambda_M^O$ is a simple closed curve.
	
	\medskip
	\noindent\textbf{Case 3: $-O\notin N$.}
	It is easy to see that there exist exactly two directions $\xi_\pm\in S_OM$ such that the half great circle
	$$
	C_{\xi_\pm}^{\mathrm{half}}:=\gamma_{(O,\xi_{\pm})}([0,\pi])
	$$
	is tangent to $\partial N$. The two tangents $\xi_{\pm}$ partition $S_OM$ into two arcs $I_0$ and $I_1$. The extended spherical geodesic $\gamma^B_{(O,\xi)}$ for each $\xi\in I_0$ does not intersect $N$, whereas the extended spherical geodesic $\gamma^B_{(O,\xi)}$ for each $\xi\in I_1$ intersects $N$ only once. The same argument as in Case~2 shows that \(\Lambda_M^O\) is a simple closed curve.
	
	Thus $\Lambda_M^O$ is a simple closed curve in all cases. The Jordan curve theorem therefore gives exactly two connected components of $\mathbb S^2\setminus\Lambda_M^O$, denoted by $M^{\mathrm I}$ and $M^{\mathrm{II}}$, with $M^{\mathrm{II}}$ chosen on the $\operatorname{Int}(N)$-side.
\end{proof}

Indeed, the connected component $M^{\mathrm I}$, which lies away from $N$, corresponds to the reflection shadow region in $\Omega$.

\begin{proposition}\label{lemma:reflection_shadow_visible_characterization}
	Let $ S\subset \Omega=\mathbb R^3\setminus C_P(N)$ be an observation surface. Then $ S$ lies in the reflection shadow if and only if $\pi_{\mathbb S^2}^P( S)\subset M^{\mathrm I}.$
\end{proposition}

\begin{proof}
	\cite[Section 2.3]{BabichDementevSamokishSmyshlyaev2000} and \cite{Smyshlyaev90} characterize the reflection shadow as:
	\begin{equation}\label{eqn : reflection shadow}
		\left\{\widetilde x\in\Omega:t(x,O)>\pi,\, x=\pi_{\mathbb S^2}^P(\widetilde x)\right\}
	\end{equation}
	where
	\begin{equation}\label{the t travel-time function}
		t(x,O):=\inf_{z\in\partial M}\bigl(d_{\mathbb S^2}(O,z)+d_{\mathbb S^2}(z,x)\bigr),\qquad x\in M,
	\end{equation}
	and $d_{\mathbb S^2}(\cdot,\cdot)$ denotes the spherical distance on $\mathbb S^2$. We claim that
	\begin{equation}\label{eqn: t}
		\Lambda_M^O = \{x\in M:t(x,O)=\pi\}.
	\end{equation}

	If $t(x,O)=\pi$ and $x\neq-O$, then we assume,
	\[
	t(x,O)=d_{\mathbb S^2}(O,z)+d_{\mathbb S^2}(z,x),\qquad z\in\partial M.
	\]
	The broken spherical geodesic segment from $O$ via $z$ to $x$ must be a reflected spherical ray defined in Definition~\ref{extended spherical geodesic flow}. If $x=-O\notin N$, we have $-O\in \Lambda_M^O$. Hence,
	\[
	\Lambda_M^O \supset \{x\in M:t(x,O)=\pi\}.
	\]
	
	Conversely, if $x\in\Lambda_M^{O}\setminus\{-O\}$, then there exists a reflected spherical ray $\gamma_{(O,\eta)}^B$ such that $x=\gamma_{(O,\eta)}^B(\pi)$. The length of the segment $\gamma_{(O,\eta)}^B([0,\pi])$ is equal to $t(x,O)$. If $x=-O\notin N$, $t(-O,O)=\pi$. Thus
	\[
	\Lambda_M^O \subset \{x\in M:t(x,O)=\pi\}.
	\]
	
	By \eqref{eqn: t} and Proposition~\ref{jordan curve theorem}, we obtain
	\begin{equation}\label{eq:M_I}
		M^{\mathrm I}=\{x\in M: t(x,O)>\pi\}.
	\end{equation}
	In view of \eqref{eqn : reflection shadow},  $ S$ lies in the reflection shadow if and only if $\pi_{\mathbb S^2}^P( S)\subset M^{\mathrm I}$.
\end{proof}

\section{Determination of spherical wavefronts}
\label{sec:determination_of_spherical_wavefronts}

Now we resume the inversion process, and turn to the reconstruction of  spherical wavefronts from the local knowledge of spherical half waves retrieved in Proposition \ref{prop:diffraction_to_half_wave}. The main result is 
 \begin{proposition} 
 	\label{prop:half_wave_to_wavefront}
 	Let
 	$
 	M_{(j)}=\mathbb S^2\setminus N_{(j)},
 	$ for $
 	j=1,2
 	$ be as  in Proposition \ref{prop:diffraction_to_half_wave}.
 	If \eqref{eqn : sine half wave on S} holds on a nonempty open subset of $M_{(1)}^{\mathrm I}\cap M_{(2)}^{\mathrm I}$,
 	then $
 		\Lambda_{M_{(1)}}^O=\Lambda_{M_{(2)}}^O.
 $
 \end{proposition}
We split the proof into three sequential steps: first, we establish analyticity of the spherical sine half wave  $u_{\mathrm{half}}^{\sin}(\pi,\cdot)$ in the region where the spherical wavefront $\Lambda_{M}^O$ has not reached yet, provided
 \begin{equation}\label{spherical sine half wave propagator}
 	u_{\mathrm{half}}^{\sin}(t,x)
 	:=
 	\sin\Bigl(t\sqrt{\Delta_M+\tfrac14}\Bigr)\delta_O(x)\in\mathcal{D}^{'}(\mathbb{R}\times M);
 \end{equation}
 second, we obtain the asymptotic expansion of $u_{\mathrm{half}}^{\sin}(\pi,\cdot)$ near $\Lambda_{M}^O$; finally, the asymptotic expansion determines the maximal real-analytic continuation domain of 
 \(u_{\mathrm{half}}^{\sin}(\pi,\cdot)\), whose boundary is \(\Lambda_M^O\).

%This section collects the material on spherical geometry and diffraction coefficients needed in this paper.
%In subsection~\ref{sec:proof_of_proposition_D}, we introduce extended spherical geodesics and the $\pi$-time spherical wavefront, and explain how the latter determines the visible part of the conical boundary.
%In subsection~\ref{sec:reduction_to_spherical_problem}, we recall the Cheeger--Taylor functional calculus on the cone $\Omega=C_P(M)$ and identify the principal diffraction coefficient with the $-\pi$-time spherical half wave propagator on the exterior base $M$.

%\subsection{Reconstruction of the visible part from the spherical wavefront}
%\label{sec:proof_of_proposition_D}

\subsection{Analytic regularity of spherical sine half waves}\label{The_analytic_singular_support_of_wave_equation}

In this subsection we study the analytic singular support of  $u_{\mathrm{half}}^{\sin}$ and show, in particular, that it is real-analytic near $\{\pi\}\times M^{\mathrm I}$. The analytic singular support of $u\in\mathcal{D}'(X)$ on a real-analytic manifold $X$ is defined by
$$
\mathrm{sing\,supp}_a\,u
:=
X\setminus
\{p\in X:\ u \text{ is real-analytic in a neighbourhood of } p\}.
$$
Throughout this subsection, $X=\mathbb{R}\times M$.

For $\bullet\in\{M,\mathbb{S}^2\}$, we first consider the regularized family
\begin{equation}\label{eq:the regularized family of S}
	U_n^{\bullet}(t,x):=
	\cos\Bigl(t\sqrt{\Delta_\bullet+\tfrac14}\Bigr)
	(\Delta_\bullet+\tfrac14)^{-n-\frac12}\delta_O(x)\in\mathcal{D}^{'}(\mathbb{R}\times \bullet), 
\end{equation}  which solve 	\begin{equation}\label{eq:solves_the_wave_equation}
\Bigl(\partial_{tt}+\Delta_{\bullet}+\tfrac14\Bigr)U_n^{\bullet}(t,x)=0.
\end{equation}
In particular, $u_{\mathrm{half}}^{\sin}$ is recovered by
$$
u_{\mathrm{half}}^{\sin}(t,x)=(-1)^{1+n}(\partial_t)^{2n+1}U_n^{M}(t,x).
$$

\begin{proposition}\label{WF(U)}
	If \(x\in M\) and \(d_{\mathbb S^2}(x,O)<s<t(x,O)\), with \(t(\cdot,O)\) as in \eqref{the t travel-time function}, namely, after the incident wavefront of $U_n^{M}$ has passed through $x$ and before the reflected wavefront of $U_n^{M}$ reaches $x$, then 
	\begin{equation*}\label{eqn:WF(U)}
		(s,x)\notin \mathrm{sing\,supp}_a\,U_n^{M}.
	\end{equation*}
\end{proposition}

\begin{proof}
	Fix $x\in M$ and $s_0$ such that $$d_{\mathbb{S}^2}(x,O)<s_0<t(x,O).$$
	Choose $\delta>0$ and a neighbourhood $U_x\ni x$ such that
	$$
	s_0+2\delta<t(U_x,O):=\inf_{y\in U_x}t(y,O).
	$$
	Hence, for $s\in (s_0-\varepsilon,s_0+\varepsilon)$ and $|t|<\varepsilon<\delta$, we have
	\begin{equation}\label{eq:finite propagation distance}
		s+|t|< t(y,O),
		\qquad y\in \widetilde U_x,
	\end{equation}
	where $\widetilde U_x\Subset M$ is an open neighbourhood of $\overline{U}_x$.

	The explicit formula in \cite[p.~115, Eq.~(4.15)]{T_PDE_2} shows that \(U^{\mathbb S^2}_0\) is real-analytic near \((s_0,x)\). In light of \eqref{eq:the regularized family of S} and \eqref{eq:solves_the_wave_equation}, we derive that
	\begin{eqnarray*}
		\Bigl(\tfrac12\Delta_{\mathbb{R}\times \mathbb{S}^2}+\tfrac18\Bigr)^nU_n^{\mathbb{S}^2}
		&=&
		\Bigl(-\tfrac12\partial_{tt}+\tfrac12\Delta_{\mathbb{S}^2}+\tfrac18\Bigr)^nU_n^{\mathbb{S}^2}\\
		&=&
		(\Delta_{\mathbb{S}^2}+\tfrac14)^nU_n^{\mathbb{S}^2}\\
		&=&
		(\Delta_{\mathbb{S}^2}+\tfrac14)^{-\frac12}\cos\Bigl(t\sqrt{\Delta_{\mathbb{S}^2}+\tfrac14}\Bigr)\delta_O
		=
		U_0^{\mathbb{S}^2}.
	\end{eqnarray*}
	The operator $\bigl(\tfrac12\Delta_{\mathbb{R}\times \mathbb{S}^2}+\tfrac18\bigr)^n$ is elliptic and has real-analytic coefficients. Hence, by analytic elliptic regularity \cite[Theorem 9.5.1.]{Ho_PDO_1}, $U_n^{\mathbb{S}^2}$ is real-analytic near $(s_0,x)$.

	It remains to show that \(U_n^M-U_n^{\mathbb S^2}\) is real-analytic near \((s_0,x)\). Together with the analyticity of  \(U_n^{\mathbb S^2}\), this will imply that \(U_n^M\) is real-analytic near \((s_0,x)\).

	We first decompose $U_n^{\bullet}$ for $\bullet\in\{M,\mathbb{S}^2\}$ into a low frequency term $U_{n,\varepsilon}^{\bullet,\mathrm{loc},l}$ and a high frequency term $U_{n,\varepsilon}^{\bullet,\mathrm{tail},l}$ such that $U_{n,\varepsilon}^{M,\mathrm{loc},l}=U_{n,\varepsilon}^{\mathbb{S}^2,\mathrm{loc},l}$. To achieve this, we invoke the following analytic tools. One is a family of cutoff functions, constructed in \cite[Theorem~1.4.2]{Ho_PDO_1}, 
	\begin{equation}\label{eq:a family of cutoff function}
		\left\{\chi_\varepsilon^l\right\}_{\varepsilon\in(0,1),\,l\in\mathbb{Z}_+}\subset C_c^\infty(\mathbb{R})
	\end{equation}
	such that:
	\begin{itemize}
		\item $\chi_\varepsilon^l(t)\equiv 1$ for $t\in[-\varepsilon/2,\varepsilon/2];$
		\item $\supp(\chi_\varepsilon^l)\subset[-\varepsilon,\varepsilon];$
		\item there exists $C\ge 1$ such that for every $l\in\mathbb{Z}_+$,
		\begin{equation}\label{cutoff}
			\|D^\alpha\chi_\varepsilon^l\|_{L^\infty(\mathbb{R})}\le \Bigl(\frac{Cl}{\varepsilon}\Bigr)^{|\alpha|},
			\qquad |\alpha|\le l .
		\end{equation}
	\end{itemize}
	The other is the function $g_n(\tau):=(\tau^2+\tfrac14)^{-n-\frac12}$ satisfying
	\begin{equation*}\label{fourier transform}
		\widehat g_n(t)=\int_{-\infty}^{\infty}g_n(\tau)e^{i\tau t}\,d\tau
		=\frac{2\cdot 4^n n!}{(2n)!}\,|t|^n\,K_n\Bigl(\frac{|t|}{2}\Bigr),
	\end{equation*}
	where $K_n$ denotes the modified Bessel function of the second kind of order $n$. 
	
	Using the functional calculus for \((\Delta_\bullet+\tfrac14)^{-n-\frac12}\), as in \cite[p.~252, Eq.~(5.51)]{T_PDE_1}, we decompose
	\begin{equation}\label{eq:split_microlocal}
		U_n^{\bullet}(s,\cdot)=U_{n,\varepsilon}^{\bullet,\mathrm{loc},l}(s,\cdot)+U_{n,\varepsilon}^{\bullet,\mathrm{tail},l}(s,\cdot),\quad\text{for $s$ near $s_0$ and any $l\in\mathbb N,$}
	\end{equation}
	where
	\begin{equation*}
		U_{n,\varepsilon}^{\bullet,\mathrm{loc},l}(s,\cdot)
		:=
		c_n\int_{\mathbb R}\chi_\varepsilon^l(t)\,|t|^nK_n\Bigl(\frac{|t|}{2}\Bigr)\,
		\cos\bigl(t\sqrt{\Delta_{\bullet}}\bigr)
		\cos\Bigl(s\sqrt{\Delta_{\bullet}+\tfrac14}\Bigr)\delta_O\,dt,
	\end{equation*}
	with $c_n:=\tfrac{4^n n!}{\pi(2n)!}$.

	First, we claim the coincidence of the low frequency terms 
	\begin{equation}\label{eq:local_terms_coincide}
		U_{n,\varepsilon}^{M,\mathrm{loc},l}
		=
		U_{n,\varepsilon}^{\mathbb{S}^2,\mathrm{loc},l}
		\qquad\text{on }(s_0-\varepsilon,s_0+\varepsilon)\times U_x.
	\end{equation}
	This follows from the finite propagation speed property established in \cite[Theorem~2.4]{CT82a} and from \eqref{eq:finite propagation distance}:
	$$
	\cos\Bigl(s\sqrt{\Delta_M+\tfrac14}\Bigr)\delta_O(x)
	=
	\cos\Bigl(s\sqrt{\Delta_{\mathbb{S}^2}+\tfrac14}\Bigr)\delta_O(x),\quad \text{on $(s_0-\varepsilon,s_0+\varepsilon)\times \widetilde U_x$.}
	$$
	Since $\Delta_M=\Delta_{\mathbb{S}^2}$ on $\widetilde U_x$, we have
	$$
	\cos\bigl(t\sqrt{\Delta_M}\bigr)\cos\Bigl(s\sqrt{\Delta_M+\tfrac14}\Bigr)\delta_O(y)
	=
	\cos\bigl(t\sqrt{\Delta_{\mathbb{S}^2}}\bigr)\cos\Bigl(s\sqrt{\Delta_{\mathbb{S}^2}+\tfrac14}\Bigr)\delta_O(y)
	$$
	for all $|t|<\varepsilon$, $s\in (s_0-\varepsilon,s_0+\varepsilon)$, and $y\in U_x$. This proves \eqref{eq:local_terms_coincide}.
	
	Second, we prove that the difference of the high frequency terms
	\(U_{n,\varepsilon}^{M,\mathrm{tail},l}-U_{n,\varepsilon}^{\mathbb S^2,\mathrm{tail},l}\)
	is real-analytic on \((s_0-\varepsilon,s_0+\varepsilon)\times U_x\).  To prove this, we need the following regularity estimate for $ U_{n,\varepsilon}^{\bullet,\mathrm{tail},l}$.
	\begin{lemma}\label{lemma:KN_goal_final}
		There exists a constant $L_0(\varepsilon)>0$ such that for all $q\in\mathbb{Z}_+$,
		\begin{equation}\label{eq:KN_goal_final}
			\|\Delta_M^q U_{n,\varepsilon}^{\bullet,\mathrm{tail},2q+6}\|_{L^2((s_0-\delta,s_0+\delta)\times U_x)}
			\le (2q)!\,L_0(\varepsilon)^{q+1},\quad \bullet\in\{M,\mathbb{S}^2\}.
		\end{equation}
	\end{lemma}
	The proof is given in the appendix. Assuming Lemma~\ref{lemma:KN_goal_final} temporarily, we complete the proof.  By \eqref{eq:split_microlocal} and \eqref{eq:local_terms_coincide},
	\begin{equation}\label{reduction}
		U_n^{M}-U_n^{\mathbb{S}^2}
		=
		U_{n,\varepsilon}^{M,\mathrm{tail},l}-U_{n,\varepsilon}^{\mathbb{S}^2,\mathrm{tail},l}
		\quad\text{on }(s_0-\varepsilon,s_0+\varepsilon)\times U_x.
	\end{equation}
	Since  $\Delta_M=\Delta_{\mathbb{S}^2}$ on $U_x$, the regularity estimates \eqref{eq:KN_goal_final}, together with \eqref{reduction} and \eqref{eq:solves_the_wave_equation}, show that
	\begin{eqnarray}\label{eq:KN_estimate_local}
		\lefteqn{\|\Delta_{\mathbb{R}\times M}^q(U_n^{M}-U_n^{\mathbb{S}^2})\|_{L^2((s_0-\delta,s_0+\delta)\times U_x)}}\nonumber\\
		&=&
		\|(2\Delta_M+\tfrac14)^q(U_n^{M}-U_n^{\mathbb{S}^2})\|_{L^2((s_0-\delta,s_0+\delta)\times U_x)}\nonumber\\
		&=&
		\|(2\Delta_M+\tfrac14)^q(U_{n,\varepsilon}^{M,\mathrm{tail},2q+6}-U_{n,\varepsilon}^{\mathbb{S}^2,\mathrm{tail},2q+6})\|_{L^2((s_0-\delta,s_0+\delta)\times U_x)}\nonumber\\
		&\le&
		(2q)!\,L_1(\varepsilon)^{q+1},
	\end{eqnarray}
	for some constant $L_1(\varepsilon)>0$. Since \(\mathbb S^2\) is real-analytic, we may choose a real-analytic local coordinate chart on \(U_x\). By \cite[Theorem~1]{Analytic_properties_of_elliptic_and_conditionally_elliptic_operators} and \eqref{eq:KN_estimate_local}, \(U_n^M-U_n^{\mathbb S^2}\) is real-analytic in a neighbourhood of \((s_0,x)\). 
\end{proof}

As an immediate consequence, $u_{\mathrm{half}}^{\sin}$ is real-analytic near $\{\pi\}\times M^{\mathrm I}$.

\begin{corollary}\label{WF(u_half)}
	Under the assumptions and notation of Proposition~\ref{jordan curve theorem}, $u_{\mathrm{half}}^{\sin}(t,x)$ is real-analytic in a neighbourhood of $\{\pi\}\times M^{\mathrm I}$.
\end{corollary}

\begin{proof}
	Since $
	u_{\mathrm{half}}^{\sin}=(-1)^{1+n}(\partial_t)^{2n+1}U_n^{M},$ Proposition~\ref{WF(U)} implies that
	$$
	(s,x)\notin \mathrm{sing\,supp}_a\,u_{\mathrm{half}}^{\sin}
	\qquad \mbox{whenever $d_{\mathbb{S}^2}(x,O)<s<t(x,O)$}  
.
	$$
	
	Now let $x\in M^{\mathrm I}$. By \eqref{eq:M_I}, we have $t(x,O)>\pi.$ Since $-O\in \Lambda_M^O\cup N$, it also follows that $x\neq -O$ and $d_{\mathbb{S}^2}(x,O)<\pi$. Then we use Proposition~\ref{WF(U)} to have
	$$
	(\pi,x)\notin \mathrm{sing\,supp}_a\,u_{\mathrm{half}}^{\sin}.
	$$
	Since $x\in M^{\mathrm I}$ was arbitrary, we conclude that
	$$
	\bigl(\{\pi\}\times M^{\mathrm I}\bigr)\cap \mathrm{sing\,supp}_a\,u_{\mathrm{half}}^{\sin}=\emptyset.
	$$
	Hence $u_{\mathrm{half}}^{\sin}(t,x)$ is real-analytic in a neighbourhood of $\{\pi\}\times M^{\mathrm I},$ as claimed.
\end{proof}

\subsection{Asymptotic behavior of spherical sine half waves}\label{Asymptotic_behavior_of_the_half wave_propagator}

In this subsection we study the asymptotic expansion at time $t=\pi$ of the spherical sine half wave  \eqref{spherical sine half wave propagator} near the spherical wavefront $\Lambda_M^O$.

To achieve this, we apply $(\Delta_M+\tfrac14)^{1/2}$ to the corresponding asymptotic expansion of the spherical sine wave 
$$
u_{\sin}(t,x)
:=
(\Delta_M+\tfrac14)^{-1/2}
\sin\Bigl(t\sqrt{\Delta_M+\tfrac14}\Bigr)\delta_O(x)
$$
near $\Lambda_M^O\setminus\{-O\}$. This is done by means of the reflected Hadamard parametrix constructed in \cite[Section~17.4]{Ho_PDO_3} and \cite{CT82b}.

\cite[(5.35)]{CT82b} derives the asymptotic formula for the spherical cosine wave
$$
u_{\cos}(t,x)
:=
\cos\Bigl(t\sqrt{\Delta_M+\tfrac14}\Bigr)\delta_O(x)
$$
near \(t=\pi\) at each point \(x\in \Lambda_M^O\setminus\{-O\}\subset M\) satisfying:
\begin{itemize}
	\item exactly one reflected spherical ray passes through $x$;
	\item the exponential map is nondegenerate at $x$ before $\pi$ time.
\end{itemize}
In our setting, these conditions follow from Proposition~\ref{one ray intersection} and Lemma~\ref{spherical geodesic non-degeneracy}. Differentiating $(\Delta_M+\tfrac14)^{-1}u_{\cos}(t,x)$ yields the following asymptotic formula for the spherical sine wave $u_{\sin}(\pi,\cdot)$ near $\Lambda_M^O\setminus\{-O\}$.

\begin{lemma}\label{the Hadamard parametrix for sine wave propagator}Let \(x\in \Lambda_M^O\setminus\{-O\}\). Then there exists a neighbourhood \(U_x\ni x\) such that, for any \(m>0\),
	\begin{equation*}\label{sine wave asymptotic}
		u_{\sin}(\pi,y)\pmod{H^m(U_x)}=
		\begin{cases}
			0, & y\in M^{\mathrm I}\cap U_x,\\[4pt]
			\Gamma(y)^{-1/2}\displaystyle\sum_{n=0}^{\mu}A_n(y)\,\Gamma(y)^n,
			& y\in M^{\mathrm {II}}\cap U_x,
		\end{cases}
		,
	\end{equation*}
	for $\mu>m+1/2$. Here
	$$
	\Gamma(y):=|\pi^2-\widetilde t^{\,2}|,
	\qquad
	\text{for } y=\gamma_{(O,\widetilde\eta)}^B(\widetilde t)\in U_x
	\text{ with }
	\widetilde t>t_B(\widetilde\eta),
	$$
	the leading coefficient $A_0$ is nowhere vanishing on $U_x$, and $A_n\in C^\infty(U_x)$ for $n\ge 0$. Equivalently,   the Fourier transform of homogeneous distributions in \cite[Eq.~(6.72)]{Ze17} yields
	\begin{equation}\label{eq:final-FIO}
		u_{\sin}(\pi,y)=
		\sum_{n=0}^{\mu}
		\int_{\mathbb R}
		 a_n(y)\,(\sigma+i0)^{-1/2-n}
		e^{-i(\pi^2-\widetilde t(y)^2)\sigma}\,d\sigma
		\pmod{H^m(U_x)}.
	\end{equation}
\end{lemma}

We now pass from the spherical sine wave to the spherical sine half wave.

\begin{proposition}\label{half wave asymptotic prop}Let \(x\in \Lambda_M^O\setminus\{-O\}\). Then there exists a neighbourhood \(U_x\ni x\) such that, for any \(m>0\),
	\begin{equation}\label{half wave asymptotic}
		u_{\mathrm{half}}^{\sin}(\pi,y)
		\pmod{H^{m-1}(U_x)}=
		\left\{
		\begin{aligned}
			&\Gamma(y)^{-3/2}\sum_{n=0}^{\mu}B_n(y)\,\Gamma(y)^n && y\in M^{\mathrm I}\cap U_x,\\
			&0 && y\in M^{\mathrm{II}}\cap U_x,
		\end{aligned}
		\right.,
	\end{equation}
	for $\mu>m+1/2$, where $B_0$ is nowhere vanishing on $U_x$ and $B_n\in C^\infty(U_x)$ are real-valued for all $n\ge 0$.
\end{proposition}

\begin{proof}
	Since $\sqrt{\Delta_\bullet+\tfrac{1}{4}}$ for $\bullet\in\{M,\mathbb{S}^2\}$ is a nonlocal operator, we need to localize $u_{\mathrm{half}}^{\sin}(\pi,y)$ in a neighbourhood $U$ of $U_x$. We achieve this as follows: for any open sets \(U_0,U\subset M\) such that
		\[
		\overline{U_x}\subset U_0\Subset U\Subset M,
		\]
		we choose cutoff functions \(\chi_U\in C_c^\infty(U)\) and \(\chi_\varepsilon\in C_c^\infty(\mathbb R)\) satisfying
		\begin{align*}
			\chi_U&\equiv1\qquad\text{on }U_0;\\
			\chi_\varepsilon(t)&\equiv1\qquad\text{for }|t|\leq \varepsilon/2,
			\qquad
			\supp(\chi_\varepsilon)\subset[-\varepsilon,\varepsilon],
		\end{align*}
		where $0<\varepsilon<d(U_x,M\setminus U_0).$

	For $k\in\mathbb{Z}_+$, we refine the decomposition of \eqref{eq:split_microlocal} by
	\begin{align}\label{calculate Im half wave}
		\lefteqn{(\Delta_\bullet+\tfrac14)^{-k-\frac12}u_{\sin}(\pi,\cdot)}\nonumber\\
	&	=
		c_k\underbrace{\int_{\mathbb R}
			\chi_\varepsilon(t)|t|^kK_k\Bigl(\frac{|t|}{2}\Bigr)
			\cos\bigl(t\sqrt{\Delta_\bullet}\bigr)
			\bigl[\chi_Uu_{\sin}(\pi,\cdot)\bigr]\,dt}_{\mathrm{ I}_\bullet}
		\nonumber\\
	&	+
		c_k\underbrace{\int_{\mathbb R}
			\chi_\varepsilon(t)|t|^kK_k\Bigl(\frac{|t|}{2}\Bigr)
			\cos\bigl(t\sqrt{\Delta_\bullet}\bigr)
			\bigl[(1-\chi_U)u_{\sin}(\pi,\cdot)\bigr]\,dt}_{\mathrm {II}_\bullet}
		\nonumber\\
	&	+
		c_k\underbrace{\int_{\mathbb R}
			(1-\chi_\varepsilon(t))|t|^kK_k\Bigl(\frac{|t|}{2}\Bigr)
			\cos\bigl(t\sqrt{\Delta_\bullet}\bigr)
			u_{\sin}(\pi,\cdot)\,dt}_{\mathrm {III}_\bullet}.
	\end{align}
	By the choice of \(\varepsilon\) and finite propagation speed, $\mathrm {II}_\bullet\equiv0$ $\text{on }U_x.$
	By the asymptotic expansion of $K_k$ in \cite[p.~202, Eq.~(1)]{Watson44}, we have
	\[
	(1-\chi_\varepsilon(t))|t|^k
	K_k\Bigl(\frac{|t|}{2}\Bigr)\in\mathcal S(\mathbb R),
	\]
	and thus $\mathrm {III}_\bullet\in C^\infty(M).$
	
	We next prove $\mathrm{I}_{\mathbb{S}^2}= \mathrm{I}_M$. Since \(\chi_Uu_{\sin}(\pi,\cdot)\) has compact support in \(M\), we extend it by zero to a distribution on \(\mathbb S^2\). Hence finite propagation speed and
	the identity \(\Delta_M=\Delta_{\mathbb S^2}\) on \(M\) imply
	\[
	\cos\bigl(t\sqrt{\Delta_{\mathbb S^2}}\bigr)
	\bigl[\chi_Uu_{\sin}(\pi,\cdot)\bigr]
	=
	\cos\bigl(t\sqrt{\Delta_M}\bigr)
	\bigl[\chi_Uu_{\sin}(\pi,\cdot)\bigr],\qquad\text{for }|t|<\epsilon,
	\]
	on \(U_x\).  Consequently, $\mathrm{I}_{\mathbb{S}^2}= \mathrm{I}_M.$
	
	Combining this with \eqref{calculate Im half wave}, we obtain
	\[
	(\Delta_M+\tfrac14)^{-k-\frac12}u_{\sin}(\pi,\cdot)
	=
	(\Delta_{\mathbb S^2}+\tfrac14)^{-k-\frac12}
	\bigl[\chi_Uu_{\sin}(\pi,\cdot)\bigr]
	\pmod{C^\infty(U_x)}.
	\]
	Applying the local differential operator
	\((\Delta_\bullet+\tfrac14)^{k+1}\) and using 
	\(\Delta_M=\Delta_{\mathbb S^2}\) on \(U_x\), we obtain
	\begin{equation*}\label{turn into S}
		u_{\mathrm{half}}^{\sin}(\pi,\cdot)
		=
		(\Delta_{\mathbb S^2}+\tfrac14)^{1/2}
		\bigl[\chi_Uu_{\sin}(\pi,\cdot)\bigr]
		\pmod{C^\infty(U_x)}.
	\end{equation*}
	By \eqref{eq:final-FIO},
	\[
	\chi_Uu_{\sin}(\pi,\cdot)(y)
	=
	\sum_{n=0}^{\mu}
	\int_{\mathbb R}
	\chi_U(y)a_n(y)(\sigma+i0)^{-1/2-n}
	e^{-i(\pi^2-\widetilde t(y)^2)\sigma}\,d\sigma
	\pmod{H^m(U_x)}.
	\]
	Since $(\Delta_{\mathbb S^2}+\tfrac14)^{1/2}
	\in\Psi^1(\mathbb S^2)$ has principal symbol \(|\xi|_{\mathbb S^2}\), the action of
	pseudodifferential operators on conormal distributions
	\cite[Theorem~18.2.12]{Ho_PDO_3} yields
	\begin{multline*}\label{FIO representation for half wave propagator}
		u_{\mathrm{half}}^{\sin}(\pi,\cdot)(y)
		=(\Delta_{\mathbb S^2}+\tfrac14)^{1/2}
		\bigl[\chi_Uu_{\sin}(\pi,\cdot)\bigr](y)\\
		=
		\sum_{n=0}^{\mu}
		\int_{\mathbb R}
		b_n(y)|\sigma|(\sigma+i0)^{-1/2-n}
		{}
		e^{-i(\pi^2-\widetilde t(y)^2)\sigma}\,d\sigma
		\pmod{H^{m-1}(U_x)},
	\end{multline*}
	where \(b_0\) is nowhere vanishing on \(U_x\) and $b_n\in C^\infty(U_x)$ are real-valued for all $n\ge 0$.
	
	Since for every \(n\in\mathbb Z\),
	\[
	|\sigma|(\sigma+i0)^{-1/2-n}
	=
	(\sigma-i0)^{1/2-n},
	\]
	the Fourier
	transform of $(\sigma+i0)^{1/2-n}$ in
	\cite[Eq.~(6.72)]{Ze17} gives
	\begin{align*}
		u_{\mathrm{half}}^{\sin}(\pi,\cdot)(y)
		&=
		\overline{
			\sum_{n=0}^{\mu}
			\int_{\mathbb R}
			b_n(y)(\sigma+i0)^{1/2-n}
			e^{-i(\widetilde t(y)^2-\pi^2)\sigma}\,d\sigma
		}\\
		&=
		\sum_{n=0}^{\mu}
		B_n(y)
		\bigl(\widetilde t(y)^2-\pi^2\bigr)_+^{-3/2+n}
		\pmod{H^{m-1}(U_x)}.
	\end{align*}
	This proves \eqref{half wave asymptotic}.
\end{proof}

\subsection{Uniqueness inversion of spherical wavefronts}
\label{sec:proof_of_proposition_C}

We now prove that the time-$\pi$ spherical sine half wave determines the spherical wavefront. The proof uses two facts established above: the propagator is real-analytic on $M^{\mathrm I}$, while along $\Lambda_M^O$ it admits an asymptotic expansion with blow-up behavior from the $M^{\mathrm I}$-side.

\begin{proof}[Proof of Proposition~\ref{prop:half_wave_to_wavefront}]
	Let $U\subset M_{(1)}^{\mathrm I}\cap M_{(2)}^{\mathrm I}$ be a
	nonempty open set on which \eqref{eqn : sine half wave on S} holds, and
	let $C$ be a connected component of
	$M_{(1)}^{\mathrm I}\cap M_{(2)}^{\mathrm I}$ such that
	$U\cap C\neq\emptyset$.
	By Corollary~\ref{WF(u_half)}, for each $i=1,2$,
	$$
	x\longmapsto
	\sin\Bigl(\pi\sqrt{\Delta_{M_{(i)}}+\tfrac14}\Bigr)\delta_O(x)
	$$
	is real-analytic in $M_{(i)}^{\mathrm I}$. Hence, by the identity theorem
	for real-analytic functions, \eqref{eqn : sine half wave on S} extends to
	$C$.
	
	To conclude the proof, it suffices to prove
	$M_{(1)}^{\mathrm I}= M_{(2)}^{\mathrm I}$. Without loss of generality, we
	show that $M_{(2)}^{\mathrm I}\subset M_{(1)}^{\mathrm I}.$ Assume, for
	contradiction, that there exists
	$$
	p\in M_{(2)}^{\mathrm I}\setminus M_{(1)}^{\mathrm I}.
	$$
	Since $M_{(1)}^{\mathrm I}\cap M_{(2)}^{\mathrm I}\neq \emptyset$ and $M_{(2)}^{\mathrm I}$ is a connected open set by Proposition~\ref{jordan curve theorem}, we must have
	$$
	\bigl(\Lambda_{M_{(1)}}^{O}\setminus\{-O\}\bigr)\cap M_{(2)}^{\mathrm I}\cap \overline{C}\neq\emptyset.
	$$
	Choose $x\in \bigl(\Lambda_{M_{(1)}}^{O}\setminus\{-O\}\bigr)\cap M_{(2)}^{\mathrm I}\cap \overline{C}.$  Since $M_{(2)}^{\mathrm I}$ is open and
	$x\notin \Lambda_{M_{(2)}}^{O}$, there exists a sufficiently small
	neighbourhood $U_x\subset M_{(2)}^{\mathrm I}$ of $x$ such that
	$U_x\cap \Lambda_{M_{(2)}}^{O}=\emptyset.$ Then
	Corollary~\ref{WF(u_half)} implies that
	\begin{equation}\label{eq:smoothness_of_second_kernel_near_x}
		\sin\Bigl(\pi\sqrt{\Delta_{M_{(2)}}+\tfrac14}\Bigr)\delta_O
		\in C^\infty(U_x).
	\end{equation}
	On the other hand, since
	$x\in \Lambda_{M_{(1)}}^{O}\setminus\{-O\},$ the asymptotic formula
	\eqref{half wave asymptotic} shows that
	$$
	\Bigl|
	\sin\Bigl(\pi\sqrt{\Delta_{M_{(1)}}+\tfrac14}\Bigr)\delta_O(y)
	\Bigr|
	\to +\infty
	\qquad
	\mbox{as $y\to x$ for $y\in C\cap U_x$}.
	$$
	However, the two distributions coincide on $C$, contradicting
	\eqref{eq:smoothness_of_second_kernel_near_x}.
\end{proof}

\section{Recovery of spherical obstacles}\label{sec : spher obst}

Finally, we show that the reconstructable part of $\partial M$ is uniquely determined by $\Lambda_M^O$. 
As illustrated in Figure~\ref{fig:boundary_recovery_visible}, each point of the spherical wavefront together with its normal direction determines a unique reflection point on $\partial M$, which in turn lifts to the visible part of the cone. 
\begin{figure}[htbp]
	\centering
	\includegraphics[width=0.8\textwidth]{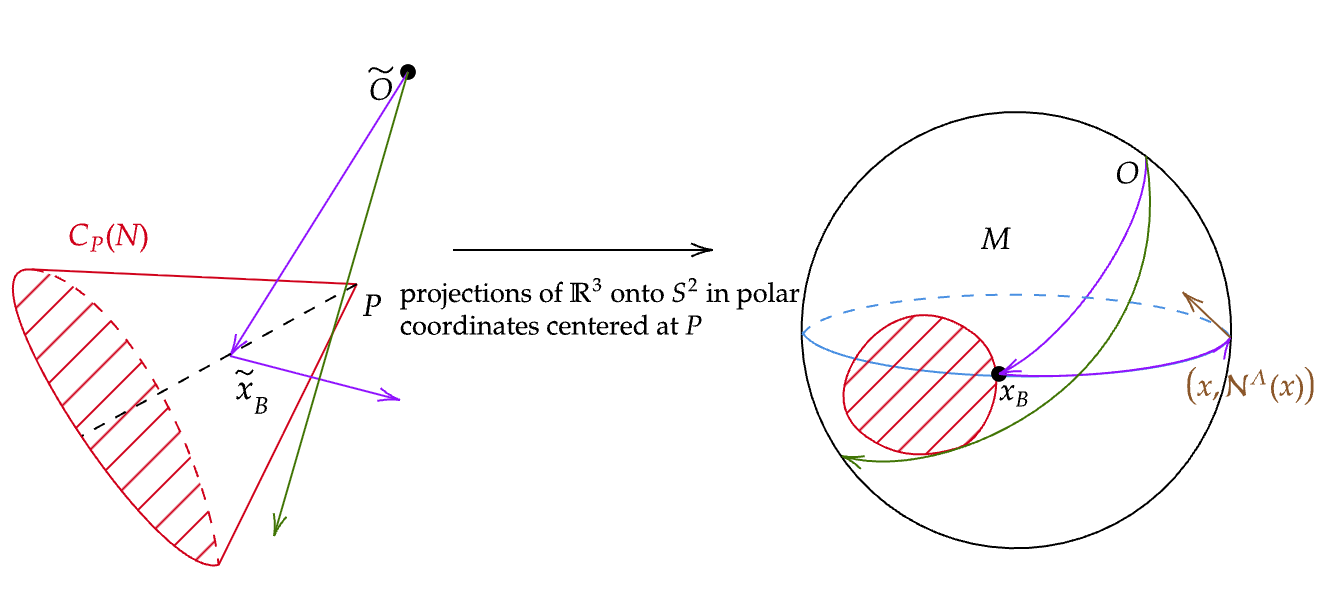}
	\caption{The spherical wavefront determines the visible part by radial projection.}
	\label{fig:boundary_recovery_visible}
\end{figure}

\begin{proposition}\label{reconstruct_boundary}
	Let $\mathcal N^\Lambda$ be the unit normal vector field along $\Lambda_M^O\setminus\{-O\}$ pointing into $M^{\mathrm I}$. For each	$
	x\in \Lambda_M^O\setminus\{-O\}$, suppose
	 $\gamma^B_{(O,\xi_x)}$ is the unique reflected spherical ray   such that $x=\gamma^B_{(O,\xi_x)}(\pi).$
Then the reflection point map
	\[
	\mathcal{R}:\Lambda_M^O\setminus\{-O\}\longrightarrow \partial M,
	\qquad
	\mathcal{R}(x):=\gamma^B_{(O,\xi_x)}\bigl(t_B(\xi_x)\bigr).
	\] is well-defined, and the boundary point $\mathcal{R}(x)$ is uniquely determined by the pair $\bigl(x,\mathcal N^\Lambda(x)\bigr).$
\end{proposition}

\begin{proof}
  Since $x\neq -O$,  $\gamma^B_{(O,\xi)}$ is a reflected spherical geodesic. By Proposition~\ref{one ray intersection}, such a direction $\xi$ is unique. This  shows that  $\mathcal{R}$ is well-defined.
	
	We next prove that the pair $\bigl(x,\mathcal N^\Lambda(x)\bigr)$ uniquely determines the reflection point $\mathcal{R}(x)$. First, choose a spherical isometry so that
	\[
	\bigl(x(\xi),\dot\gamma^B_{(O,\xi)}(\pi)\bigr)
	=
	\bigl((1,0,0),(0,-1,0)\bigr).
	\]
	In this coordinate system, the reflection point lies on the equator and is given by
	\[
	x_B(\xi)=(\cos\theta,\sin\theta,0),\qquad \theta\in(0,\pi).
	\]	
	We then express the source point in the same coordinate system as
	\[
	O=(r_0\cos\theta_0,r_0\sin\theta_0,\sqrt{1-r_0^2}),\qquad 0\le r_0\le 1.
	\]
	In these coordinates, the travel time $T(\theta)$ from $O$ to $x_B(\xi)$ through $x(\xi)$ is 
	\[
	T(\theta)=\theta+2\arcsin\sqrt{\frac12\bigl(1-r_0\cos(\theta-\theta_0)\bigr)},\qquad \theta\in(0,\pi).
	\] 
	
	\medskip
	\noindent\textbf{Case 1: $0\le r_0<1$.}
	A direct computation shows
	\[
	\frac{d}{d\theta}T(\theta)
	=
	1+\frac{r_0\sin(\theta-\theta_0)}{\sqrt{1-r_0^2\cos^2(\theta-\theta_0)}}>0.
	\]
	Hence $T$ is strictly increasing on $(0,\pi)$. Since $T(0)<\pi<T(\pi)$, there exists a unique $\theta_1\in(0,\pi)$ such that $T(\theta_1)=\pi$.
	
	\medskip
	\noindent\textbf{Case 2: $r_0=1$.}
	In this case, $O=(\cos\theta_0,\sin\theta_0,0),$ and
	\[
	T(\theta)=\theta+|\theta-\theta_0|.
	\]
	Since $x(\xi)\neq -O$, we have $\theta_0\in(-\pi,\pi)$, and there exists a unique $\theta_1\in(0,\pi)$ such that $T(\theta_1)=\pi$.
	
	Thus in both cases the reflection point
	\[
	x_B(\xi)=(\cos\theta_1,\sin\theta_1,0)
	\]
	is uniquely determined by the pair $\bigl(x(\xi),\dot\gamma^B_{(O,\xi)}(\pi)\bigr).$
	
	Finally, we identify $\dot{\gamma}^B_{(O,\xi)}(\pi)$ with $\mathcal N^\Lambda(x)$. By \eqref{Gauss lemma} and the smoothness of \(\Lambda_M^O\setminus\{-O\}\) (Proposition~\ref{stratification}), the vector \(\dot\gamma^B_{(O,\xi)}(\pi)\) is normal to \(\Lambda_M^O\).
	Proposition~\ref{lemma:reflection_shadow_visible_characterization} gives that
	\[
	\{\gamma^B_{(O,\xi)}(t)\in M: t_B(\xi)<t<\pi\}\subset M^{\mathrm{II}},
	\]
	and $\gamma^B_{(O,\xi)}(t)$ approaches $\Lambda_M^O$ as $t\to\pi-0$.  Therefore,
	\[
	\dot\gamma^B_{(O,\xi)}(\pi)=\mathcal N^\Lambda(x(\xi)).
	\]
	Hence the pair $\bigl(x,\mathcal N^\Lambda(x)\bigr)$ uniquely determines $\mathcal{R}(x)$.
\end{proof}

\section{Detection of conic obstacles}
\label{Reconstruction_of_the_strictly_convex_cone}

 Now we are in the position to complete the proof of Theorem~\ref{main}.
 The reconstruction process follows the strategy sketched in \ref{sec : methodlogy}.

%	By Proposition~\ref{prop:diffraction_to_half_wave}, once the tip $P$ is identified, the diffraction coefficients recover the spherical half wave propagator on the exterior base $M\subset\mathbb S^2$. Then Proposition~\ref{prop:half_wave_to_wavefront}  proves that this data, measured on $\pi_{\mathbb S^2}^P(S)$, determines the full $\pi$-time spherical wavefront.

%	Proposition~\ref{prop:wavefront_to_visible_boundary} then shows that the recovered spherical wavefront determines the visible part of the conic boundary, thereby completing the proof of Theorem~\ref{main}.  
	
	\begin{proof}[Proof of Theorem~\ref{main}] 
	
	Suppose that two  strictly convex cones
	$
	C_{P_{(1)}}(N_{(1)})$ and $
	C_{P_{(2)}}(N_{(2)})
	$
	have the same local measurements.
	In the frequency domain, this means that their lens data and principal
	diffraction coefficients agree:
	$$
	\left(
	\widetilde \xi_{(1)}(\widetilde x),
	D_{0,(1)}(\widetilde x,\widetilde O)
	\right)
	=
	\left(
	\widetilde \xi_{(2)}(\widetilde x),
	D_{0,(2)}(\widetilde x,\widetilde O)
	\right),
	\qquad
	\widetilde x\in S.
	$$
	In the time domain, this means that their diffracted arrival times and
	arrival-time diffraction coefficients agree:
	$$
	\left(
	\ell_{(1)}(\widetilde x),
	d_{0,(1)}(0,\widetilde x,\widetilde O)
	\right)
	=
	\left(
	\ell_{(2)}(\widetilde x),
	d_{0,(2)}(0,\widetilde x,\widetilde O)
	\right),
	\qquad
	\widetilde x\in S.
	$$

We first spot the cone tip.
	\begin{proposition}
		\label{prop:tip_recovery}
		Let $\widetilde O\in \mathbb R^3$, $P_{(1)},P_{(2)}\in\mathbb R^3$ and   $S\subset \mathbb R^3$  as in
		Theorem~\ref{main}.   If either   the lens data of the diffracted waves agree:
			\[
			(\widetilde x,\widetilde \xi_{(1)}(\widetilde x))
			=
			(\widetilde x,\widetilde \xi_{(2)}(\widetilde x)), \quad \forall \widetilde x\in S
			\]
			or  the diffracted arrival times agree:
			\[
			\ell_{(1)}(\widetilde x)=\ell_{(2)}(\widetilde x), \quad \forall \widetilde x\in S.
			\] 
		then $P_{(1)}=P_{(2)}.$
	\end{proposition}
	
	\begin{proof}
		We first consider the lens data. Equality of the lens data implies that, for every $\widetilde x\in S$, the two tips $P_{(1)}$ and $P_{(2)}$ lie on the same line determined by $(\widetilde x,\widetilde \xi(\widetilde x))$.	
		Choose $\widetilde x_1,\widetilde x_2\in S$ such that the corresponding lines are distinct; this is possible since $S$ is not contained in a single line. The two lines must intersect at a common point, and thus $P_{(1)}=P_{(2)}$.
		
		We next consider the arrival-time data. Assume $P_{(1)}\neq P_{(2)}$. Then $\ell_{(1)}(\widetilde x)=\ell_{(2)}(\widetilde x)$ implies
		\[
		|\widetilde x-P_{(2)}|-|\widetilde x-P_{(1)}|=|\widetilde O-P_{(1)}|-|\widetilde O-P_{(2)}|=c_0,
		\quad \widetilde x\in S,
		\]
		and whence $S$ is contained in a two-foci hyperboloid. This contradicts the assumption in Theorem~\ref{main}. Hence $P_{(1)}=P_{(2)}$.
	\end{proof}

%	By Proposition~\ref{prop:tip_recovery}, either the equality of the	lens data in the frequency domain or the equality of the arrival times	in the time domain determines the diffraction point uniquely. Hence $P:=P_{(1)}=P_{(2)}.$
	
	With the tip recovered, the spherical projection $\pi_{\mathbb S^2}^{P}$ is determined. We can fix
	$
	O=\pi_{\mathbb S^2}^{P}(\widetilde O)$ and $
	x=\pi_{\mathbb S^2}^{P}(\widetilde x).
	$
	Set $M_{(j)}:=\mathbb S^2\setminus N_{(j)},$ $j=1,2.$ By Proposition~\ref{prop:diffraction_to_half_wave}, the equality of
	the diffraction coefficients, either in the frequency domain or at the
	diffracted arrival time in the time domain, implies \eqref{eqn : sine half wave on S}
	on $\pi_{\mathbb S^2}^{P}(S)
	\cap
	M_{(1)}^{\mathrm I}
	\cap
	M_{(2)}^{\mathrm I}.$
	
	Since $S$ has non-vanishing Gauss curvature, the projection
	$\pi_{\mathbb S^2}^{P}|_S$ has full rank at some point. Therefore,
	by the inverse function theorem, $\pi_{\mathbb S^2}^{P}(S)$ contains
	a nonempty open subset $U$. Since $S$ lies in the reflection shadow,
	Proposition~\ref{lemma:reflection_shadow_visible_characterization} shows that
	$
	U\subset M_{(1)}^{\mathrm I}\cap M_{(2)}^{\mathrm I}.
	$
	Hence \eqref{eqn : sine half wave on S} holds on
	a nonempty open subset of $M_{(1)}^{\mathrm I}\cap M_{(2)}^{\mathrm I}.$ Proposition~\ref{prop:half_wave_to_wavefront} then yields
	$
	\Lambda_{M_{(1)}}^O=\Lambda_{M_{(2)}}^O.
	$

	Finally, the smoothness of $\Lambda_M^O\setminus\{-O\}$ determines $\mathcal N^\Lambda$ along $\Lambda_M^O\setminus\{-O\}$. By Proposition~\ref{reconstruct_boundary},  $\Lambda_M^O\setminus\{-O\}$ and $\mathcal N^\Lambda$ recover the set
	$$
	\overline{\mathcal{R}\bigl(\Lambda_M^O\setminus\{-O\}}\bigr)\subset \partial M.
	$$
	This set is equal to the projection of visible boundary $\operatorname{Vis}_{\widetilde O}
	\left(
	\partial C_P(N)
	\right)$. Since $P$ has been detected, the two cones have the same visible boundary
	$$
	\operatorname{Vis}_{\widetilde O}
	\left(
	\partial C_P(N_{(1)})
	\right)
	=
	\operatorname{Vis}_{\widetilde O}
	\left(
	\partial C_P(N_{(2)})
	\right).
	$$
	
	This proves the injectivity of both maps
	\eqref{eqn : FD observation at l} and \eqref{eqn : observation at l}.
\end{proof}

\appendix

\section{Proofs of technical lemmas}\label{appendix:reflected_rays}

%This appendix records the geometric properties of reflected spherical rays needed in the main text.

%The next lemma shows that, along a reflected spherical ray, the  exponential map is nondegenerate.

\subsection{Non-degeneracy of spherical exponential maps}

\begin{proof}[Proof of Lemma \ref{spherical geodesic non-degeneracy}]
	We work in Cartesian coordinates on \(\mathbb R^3\), taking the center of \(\mathbb S^2\) as the origin and the plane of \(C_{x_B(\eta)}\) as the \(xy\)-plane. Write the first hitting time $t_B:=t_B(\eta)$ and the first hitting point $x_B:=x_B(\eta)$ as in Definition~\ref{the classification of ray}. Since the ray $\gamma_{(O,\eta)}^B$ meets $\partial M$ transversally, the maps $\xi\mapsto t_B(\xi)$ and $\xi\mapsto x_B(\xi)$ are smooth near $\eta$.
	
	Fix $t\in(t_B,\pi]$. After shrinking a neighbourhood $U\subset S_OM$ of $\eta$ if necessary, we have $t>t_B(\xi)$ for every $\xi\in U$. The fixed-time exponential map
	\[
	G_t(\xi):=\exp_O^{B}(t,\xi)
	\]
	is given by
	\begin{equation}\label{Gt-reflection-formula}
		G_t(\xi)
		=
		\gamma_{(O,\xi)}(t)
		-2\Big\langle \gamma_{(O,\xi)}(t),\mathcal N\bigl(x_B(\xi)\bigr)\Big\rangle_{\mathbb R^3}
		\mathcal N\bigl(x_B(\xi)\bigr),
	\end{equation}
	where $\mathcal N$ is the inward unit normal to $M$ along $\partial M$. The spherical geodesic $\gamma_{(O,\xi)}(t)$ lies on the $N$-side of the tangent great circle to $N$ at $x_B(\xi)$. Hence, after possibly shrinking $U$ further,
	\begin{equation}\label{negative}
		\Big\langle \gamma_{(O,\xi)}(t),\mathcal N\bigl(x_B(\xi)\bigr)\Big\rangle_{\mathbb R^3}<0,
		\qquad \xi\in U.
	\end{equation}
	
	We first prove that the angular differential $d_\eta G_t$ is injective. Let \(w\in T_\eta(S_OM)\) be a unit tangent vector, and define
	\[
	v:=d_\eta x_B(w)\in T_{x_B}\partial M.
	\]
	Differentiation of \eqref{Gt-reflection-formula} gives
	\begin{multline}\label{differential_readable}
		d_\eta G_t(w)
		=
		d_\eta\gamma_{(O,\cdot)}(t)(w)
		-2\big\langle \gamma_{(O,\eta)}(t),\mathcal N(x_B)\big\rangle_{\mathbb R^3}
		\nabla_v\mathcal N\\
		-2d_\eta\!\left[\big\langle \gamma_{(O,\cdot)}(t),\mathcal N\circ x_B\big\rangle_{\mathbb R^3}\right](w)
		\mathcal N(x_B),
	\end{multline}
	where $\nabla_v \mathcal N$ is the direction derivative of $\mathcal N$ along $v$. Pairing \eqref{differential_readable} with $v$ and using $\langle \mathcal N(x_B),v\rangle_{\mathbb R^3}=0$, we obtain
	\begin{equation}\label{eq:pairing-Gt-v}
			\big\langle d_\eta G_t(w),v\big\rangle_{\mathbb R^3}
			=
			\left\langle d_\eta\gamma_{(O,\cdot)}(w),v\right\rangle_{\mathbb R^3}-2\big\langle \gamma_{(O,\eta)}(t),\mathcal N(x_B)\big\rangle_{\mathbb R^3}
			\big\langle \nabla_v\mathcal N,v\big\rangle_{\mathbb R^3}.
	\end{equation}
	
	Differentiating $x_B(\xi)$ at $\xi=\eta$ in the direction $w$, we obtain
	\begin{equation}\label{eq:differential-hitting-point}
		v
		=
		\dot\gamma_{(O,\eta)}(t_B)\,d_\eta t_B(w)
		+d_\eta\gamma_{(O,\cdot)}(t_B)(w).
	\end{equation}
	By the Gauss lemma,
	\begin{equation*}
		\left\langle
		\dot\gamma_{(O,\eta)}(t),
		d_\eta\gamma_{(O,\cdot)}(t)(w)
		\right\rangle_{\mathbb R^3}
		=\left\langle
		\eta,
		w
		\right\rangle_{\mathbb R^3}=0.
	\end{equation*}
	Since $\gamma_{(O,\eta)}$ is a great circle on $\mathbb{S}^2$ which is a planar curve, $d_\eta\gamma_{(O,\cdot)}(t)(w)$ is normal to this plane. Therefore, 
	\begin{equation}\label{eq:orthogonality}
		\left\langle
		d_\eta\gamma_{(O,\cdot)}(t)(w),
		\dot\gamma_{(O,\eta)}(t_B)
		\right\rangle_{\mathbb R^3}=0.
	\end{equation}
	
	After a rotation of $\mathbb S^2$, we may assume that
	\[
	O=(0,0,1),
	\qquad
	\eta=(\cos\theta,\sin\theta,0),
	\qquad
	w=(-\sin\theta,\cos\theta,0).
	\]
	For the spherical geodesic \(\gamma_{(O,\eta)}\), we have
	\[
	\gamma_{(O,\eta)}(t)=(\sin t\cos \theta, \sin t\sin \theta, \cos t),\qquad t\in [0,\pi].
	\]
	Then
	\begin{align*}
		d_\eta\gamma_{(O,\cdot)}(t)(w)
		&=\frac{d}{ds}\gamma_{(O,(\cos(\theta+s),\sin(\theta+s),0))}(t)\Big|_{s=0}\\
		&=(-\sin t\sin\theta,\sin t\cos\theta,0).
	\end{align*}
	Since \(0< t_B<t\leq\pi\),
	\[
    \left\langle	d_\eta\gamma_{(O,\cdot)}(t)(w),
	d_\eta\gamma_{(O,\cdot)}(t_B)(w)
	\right\rangle_{\mathbb R^3}
	=
	\sin t\,\sin t_B\geq0.
	\]
	Combining this identity with \eqref{eq:differential-hitting-point} and \eqref{eq:orthogonality}, we obtain
	\begin{equation}\label{eq:free-pairing-nonnegative}
		\left\langle
		d_\eta\gamma_{(O,\cdot)}(t)(w),v
		\right\rangle_{\mathbb R^3}
		=
		\sin t\,\sin t_B\geq0.
	\end{equation}
	
	Since $\mathcal N$ is the inward normal to $M$, it is the outward normal to $N$. Thus the strict convexity in Definition~\ref{cone} implies
	\[
	\big\langle \nabla_v\mathcal N,v\big\rangle_{\mathbb R^3}>0.
	\]
	Using this inequality together with \eqref{negative}, \eqref{eq:pairing-Gt-v}, and \eqref{eq:free-pairing-nonnegative}, we conclude that
	\[
	\big\langle d_\eta G_t(w),v\big\rangle_{\mathbb R^3}
	>
	\left\langle
	d_\eta\gamma_{(O,\cdot)}(t)(w),v
	\right\rangle_{\mathbb R^3}
	\geq0.
	\]
	Hence $d_\eta G_t(w)\neq0$. Since $\dim S_OM=1$, the map
	\[
	d_\eta G_t:T_\eta(S_OM)\longrightarrow T_{G_t(\eta)}M
	\]
	is injective.
	
	It remains to show that the angular and time derivatives are linearly independent. Let $w\in T_\eta(S_OM)$. Set
	\[
	V(t):=d_\eta\exp_O^{B}(t,\cdot)(w),
	\qquad
	\dot\gamma(t):=d_t\exp_O^{B}(\cdot,\eta)(1)=
	\dot\gamma_{(O,\eta)}^B(t),
	\]
	where \(V(t)\) and \(\dot\gamma(t)\) denote the angular and time derivatives, respectively. To prove linear independence, it suffices to show that
	\[
	\big\langle V(t),\dot\gamma(t)\big\rangle_{\mathbb S^2}=0,\qquad\text{for }t\in [0,\pi].
	\]
	Differentiating $\big\langle V(t),\dot\gamma(t)\big\rangle_{\mathbb S^2}$ and using the geodesic equation together with the unit-speed parametrization gives
	\[
	\frac{d}{dt}\big\langle V(t),\dot\gamma(t)\big\rangle_{\mathbb S^2}
	=
	\frac{1}{2}d_\eta \left[\big\langle \dot\gamma(t),\dot\gamma(t)\big\rangle_{\mathbb S^2}\right](w)=0.
	\]
	Thus, $t\mapsto \big\langle V(t),\dot\gamma(t)\big\rangle_{\mathbb S^2}$ is constant on \( [0,t_B) \) and on \( (t_B,\pi] \). We show that these two constants agree.
	
	Since $x_B(\xi)\in\mathbb S^2$ and
	$\mathcal N(x_B(\xi))\in T_{x_B(\xi)}\mathbb S^2$, we have
	\[
	\big\langle x_B(\xi),\mathcal N(x_B(\xi))\big\rangle_{\mathbb R^3}=0.
	\]
	Differentiating at $\xi=\eta$ gives
	\[
	\big\langle v,\mathcal N(x_B)\big\rangle_{\mathbb R^3}
	+
	\big\langle x_B,\nabla_v\mathcal N\big\rangle_{\mathbb R^3}=0.
	\]
	The first term vanishes because $v\in T_{x_B}\partial M$; hence $\big\langle x_B,\nabla_v\mathcal N\big\rangle_{\mathbb R^3}=0.$ Therefore, by \eqref{differential_readable} and
	$\langle x_B,\mathcal N(x_B)\rangle_{\mathbb R^3}=0$,
	\begin{align*}
		V(t_B+0)
		&= d_\eta\gamma_{(O,\cdot)}(t_B)(w)
		-
		\lim_{t\downarrow t_B}
		2d_\eta\!\left[\big\langle \gamma_{(O,\cdot)}(t),\mathcal N\circ x_B\big\rangle_{\mathbb R^3}\right](w)
		\mathcal N(x_B)\\
		&=
		V(t_B-0)
		-2\big\langle V(t_B-0),\mathcal N(x_B)\big\rangle_{\mathbb R^3} \mathcal N(x_B)
		-2\big\langle x_B,\nabla_v\mathcal N\big\rangle_{\mathbb R^3}
		\mathcal N(x_B)\\
		&=
		V(t_B-0)
		-2\big\langle V(t_B-0),\mathcal N(x_B)\big\rangle_{\mathbb R^3}
		\mathcal N(x_B).
	\end{align*}
	By \eqref{eqn:reflection vector}, $\dot{\gamma}(t)$ satisfies
	\begin{equation*}
		\dot{\gamma}(t_B+0)
		=
		\dot{\gamma}(t_B-0)
		-2\langle \dot{\gamma}(t_B-0),\mathcal N(x_B)\rangle_{\mathbb R^3}
		\mathcal N(x_B).
	\end{equation*}
	Pairing up $V(t_B+0)$ and $\dot{\gamma}(t_B+0)$, we have
	\[
	\big\langle V(t_B-0),\dot\gamma(t_B-0)\big\rangle_{\mathbb S^2}
	=
	\big\langle V(t_B+0),\dot\gamma(t_B+0)\big\rangle_{\mathbb S^2}.
	\]
	Thus $t\mapsto\langle V(t),\dot\gamma(t)\rangle_{\mathbb S^2}$ is constant on $[0,\pi]$. Since $V(0)=0$, 
	\begin{equation}\label{Gauss lemma}
		\big\langle V(t),\dot\gamma(t)\big\rangle_{\mathbb S^2}=0,
		\qquad t\in[0,\pi].
	\end{equation}
	
	For fixed $t\in(t_B,\pi]$, the injectivity proved above gives
	\[
	V(t)=d_\eta G_t(w)\neq0.
	\]
	By \eqref{Gauss lemma}, $V(t)$ is orthogonal to $\dot\gamma(t)$.
\end{proof}

%\section{Proof of \eqref{eq:tail_decay_microlocal}}\label{appendix1}
%We establish \eqref{eq:tail_decay_microlocal} by repeated integration by parts in the oscillatory integral representation of \(f_\varepsilon^{2q}(\tau)\).

\subsection{Regularity of the tail term}

\begin{proof}[Proof of Lemma~\ref{lemma:KN_goal_final}]
	Consider the high frequency multiplier
	\begin{equation}\label{f_m}
		f_{\varepsilon}^{l}(\tau)
		:=c_n
		\int_{\mathbb{R}}\bigl(1-\chi_\varepsilon^l(t)\bigr)\,|t|^n
		K_n\Bigl(\frac{|t|}{2}\Bigr)\cos(t\tau)\,dt,
	\end{equation}
	so that
	$$
	U_{n,\varepsilon}^{\bullet,\mathrm{tail},l}(s,\cdot)
	=
	f_{\varepsilon}^{l}(\sqrt{\Delta_\bullet})\,
	\cos\Bigl(s\sqrt{\Delta_\bullet+\tfrac14}\Bigr)\delta_O.
	$$
	
	The desired estimate \eqref{eq:KN_goal_final} follows from the following dyadic spectral bound: there exists \(C>0\) such that for all $q\in\mathbb{Z}_+$ and all $l\in\mathbb{N}$,
	\begin{equation}\label{eq:proj_reduction_microlocal}
		\|\Delta_\bullet^q U_{n,\varepsilon}^{\bullet,\mathrm{tail},l}\|_{L^2((s_0-\delta,s_0+\delta)\times U_x)}^2
		\le C\sum_{m\in\mathbb{Z}_+} m^{2q+3}\sup_{\tau\in(m-1,m]}\bigl|f_\varepsilon^{l}(\sqrt{\tau})\bigr|^2.
	\end{equation}
	In fact, it follows from the functional calculus for $\Delta_\bullet$ that
	\begin{eqnarray}\label{L infinite estimate}
		\|\Delta_\bullet^q U_{n,\varepsilon}^{\bullet,\mathrm{tail},l}\|_{L^2((s_0-\delta,s_0+\delta)\times U_x)}^2
		&\leq&
		\left\|\|\Delta_\bullet^q U_{n,\varepsilon}^{\bullet,\mathrm{tail},l}(s,\cdot)\|_{L^2(\bullet)}\right\|_{L^2((s_0-\delta,s_0+\delta))}^2 \nonumber\\
		&\leq&
		2\delta\,\bigl\|\Delta_\bullet^q U_{n,\varepsilon}^{\bullet,\mathrm{tail},l}(0,\cdot)\bigr\|_{L^2(\bullet)}^2 \nonumber\\
		&=&
		2\delta\,\bigl\|\Delta_\bullet^q f_\varepsilon^l(\sqrt{\Delta_\bullet})\delta_O\bigr\|_{L^2(\bullet)}^2 \nonumber\\
		&=&
		2\delta\sum_{\lambda_j\in\Spec(\Delta_\bullet)}
		\bigl(\lambda_j^q f_\varepsilon^l(\sqrt{\lambda_j})\bigr)^2\,|\varphi_j(O)|^2, \nonumber
	\end{eqnarray}
	where $\varphi_j\in C^\infty(\bullet)$ is a normalized $L^2$-eigenfunction of $\Delta_\bullet$ associated with  eigenvalue $\lambda_j$. The Sobolev embedding theorem yields that
	$$
	|\varphi_j(O)| \lesssim \|\varphi_j\|_{H^2(\bullet)} \lesssim \|\Delta_\bullet\varphi_j\|_{L^2(\bullet)} +\|\varphi_j\|_{L^2(\bullet)}\lesssim \lambda_j,\quad\text{for $\lambda_j>0$}.
	$$
	Therefore, we have
	$$
	\|\Delta_\bullet^q U_{n,\varepsilon}^{\bullet,\mathrm{tail},l}\|_{L^2((s_0-\delta,s_0+\delta)\times U_x)}^2
	\lesssim
	\sum_{\lambda_j\in\Spec(\Delta_\bullet)}
	\bigl(\lambda_j^q f_\varepsilon^l(\sqrt{\lambda_j})\bigr)^2 \lambda_j^2.
	$$
	Applying Weyl's law for manifolds with piecewise smooth boundary
	\cite[p.~172]{Eigenvalues_in_Riemannian_geometry}, we obtain
	\begin{eqnarray*}
		&&
		\|\Delta_\bullet^q U_{n,\varepsilon}^{\bullet,\mathrm{tail},l}\|_{L^2((s_0-\delta,s_0+\delta)\times U_x)}^2 \nonumber\\
		&\lesssim&
		\sum_{m\in\mathbb{Z}_+}
		\#\Bigl\{\lambda\in(m-1,m] : \lambda\in\Spec(\Delta_\bullet)\Bigr\}\,
		m^{2q+2}
		\sup_{\tau\in(m-1,m]}\bigl|f_\varepsilon^l(\sqrt{\tau})\bigr|^2 \\
		&\lesssim&
		\sum_{m\in\mathbb{Z}_+}
		m^{2q+3}\,
		\sup_{\tau\in(m-1,m]}\bigl|f_\varepsilon^l(\sqrt{\tau})\bigr|^2,
	\end{eqnarray*}
	which proves \eqref{eq:proj_reduction_microlocal}.

	With regard to the estimates of $|f_\varepsilon^l(\sqrt{\tau})\bigr|$ as $\tau\to\infty$, we claim that there exist constants $C(\varepsilon),L(\varepsilon)>0$ such that for all $q\in\mathbb{Z}_+$ and all $\tau\ge 1$,
		\begin{equation}\label{eq:tail_decay_microlocal}
			\bigl|f_\varepsilon^{2q}(\tau)\bigr|
			\le C(\varepsilon)\Bigl(\frac{L(\varepsilon)q}{\tau}\Bigr)^{2q}.
	\end{equation}

To show \eqref{eq:tail_decay_microlocal}, we rewrite  \eqref{f_m} as
$$
f_\varepsilon^{2q}(\tau)
=
c_n\int_{\mathbb R}\bigl(1-\chi_\varepsilon^{2q}(t)\bigr)\,b(t)\cos(t\tau)\,dt,
\qquad
c_n:=\frac{1}{\pi}\frac{4^n n!}{(2n)!},
$$
with \begin{equation*}\label{eqn:mapping_b}
	b(t):=|t|^nK_n\Bigl(\frac{|t|}{2}\Bigr), \qquad t\in\mathbb R.
\end{equation*}

We first establish an estimate for  $b^{(m)}:=\partial_t^{m}b$. It remains to show that there exist constants $C_1(\varepsilon),L_1(\varepsilon)>0$ such that
\begin{equation}\label{eq:b_derivative_L1}
	\int_{|t|\ge \varepsilon/2}|b^{(m)}(t)|\,dt
	\le
	C_1(\varepsilon)L_1(\varepsilon)^m m!,
	\qquad m\in\mathbb Z_+.
\end{equation}
Indeed, for both of  $[\varepsilon/2,\infty)$ and $(-\infty,-\varepsilon/2]$, the Bessel function $K_n(z)$ extends holomorphically to a neighbourhood in $\mathbb{C}$ staying away from $0\in \mathbb{C}$. For $r_1(\varepsilon)>0$, Cauchy's integral formula gives
$$
b^{(m)}(t)
=
\frac{m!}{2\pi i}
\int_{|z-t|=r_1(\varepsilon)}
\frac{b(z)}{(z-t)^{m+1}}\,dz,
\qquad |t|\ge \varepsilon/2.
$$
Consequently,
$$
|b^{(m)}(t)|
\le
r_1(\varepsilon)^{-m}m!\,
\sup_{|z-t|=r_1(\varepsilon)}|b(z)|,
\qquad |t|\ge \varepsilon/2.
$$
Using the exponential decay of $K_n(z)$ in sectors $|\arg z|\le \pi-\delta$; see \cite[p.~202, Eq.~(1)]{Watson44}, there exist constants $C_2(\varepsilon),c_1(\varepsilon)>0$ such that
$$
\sup_{|z-t|=r_1(\varepsilon)}|b(z)|
\le
C_2(\varepsilon)e^{-c_1(\varepsilon)|t|},
\qquad |t|\ge \varepsilon/2.
$$
It follows that
$$
\int_{|t|\ge \varepsilon/2}|b^{(m)}(t)|\,dt
\le
C_2(\varepsilon)r_1(\varepsilon)^{-m}m!
\int_{|t|\ge \varepsilon/2}e^{-c_1(\varepsilon)|t|}\,dt.
$$
Thus, setting $L_1(\varepsilon):=r_1(\varepsilon)^{-1},$ we obtain \eqref{eq:b_derivative_L1}.

We next estimate 
$$
a_{q,\varepsilon}(t):=\bigl(1-\chi_\varepsilon^{2q}(t)\bigr)b(t),
$$
with cutoff function $\chi_\varepsilon^{2q}$ in \eqref{eq:a family of cutoff function}. We claim that there exist constants $C_3(\varepsilon),L_2(\varepsilon)>0$ such that
\begin{equation}\label{eq:a_derivative_final}
	\|a_{q,\varepsilon}^{(2q)}\|_{L^1(\mathbb R)}
	\le
	C_3(\varepsilon)\bigl(L_2(\varepsilon)q\bigr)^{2q}.
\end{equation}
Since $\chi_\varepsilon^{2q}(t)\equiv1$ on $[-\varepsilon/2,\varepsilon/2]$, we have $a_{q,\varepsilon}(t)=0$ $\text{for } |t|\le \varepsilon/2.$ By Leibniz's rule,
$$
a_{q,\varepsilon}^{(2q)}(t)
=
\sum_{j=0}^{2q}\binom{2q}{j}
\Bigl(\frac{d}{dt}\Bigr)^j
\bigl(1-\chi_\varepsilon^{2q}(t)\bigr)
b^{(2q-j)}(t).
$$
Therefore,
\begin{equation}\label{eq:Leibniz_L1}
	\|a_{q,\varepsilon}^{(2q)}\|_{L^1(\mathbb R)}
	\le
	\sum_{j=0}^{2q}\binom{2q}{j}
	\Bigl\|
	\Bigl(\frac{d}{dt}\Bigr)^j
	\bigl(1-\chi_\varepsilon^{2q}\bigr)
	\Bigr\|_{L^\infty}
	\int_{|t|\ge \varepsilon/2}
	|b^{(2q-j)}(t)|\,dt.
\end{equation}
Substituting \eqref{cutoff} and \eqref{eq:b_derivative_L1} into \eqref{eq:Leibniz_L1}, we obtain, for some constant $C_4(\varepsilon)>0$,
$$
\|a_{q,\varepsilon}^{(2q)}\|_{L^1(\mathbb R)}
\le
C_1(\varepsilon)
\sum_{j=0}^{2q}\binom{2q}{j}
\Bigl(\frac{C_4(\varepsilon)q}{\varepsilon}\Bigr)^j
L_1(\varepsilon)^{2q-j}(2q-j)!.
$$
Using the elementary inequality $(2q-j)!\le (2q)^{2q-j},$ we get
\begin{align*}
	\|a_{q,\varepsilon}^{(2q)}\|_{L^1(\mathbb R)}
	&\le
	C_1(\varepsilon)
	\sum_{j=0}^{2q}\binom{2q}{j}
	\Bigl(\frac{C_4(\varepsilon)q}{\varepsilon}\Bigr)^j
	\bigl(2L_1(\varepsilon)q\bigr)^{2q-j}\\
	&=C_1(\varepsilon)
	\Bigl(
	2L_1(\varepsilon)q+\frac{C_4(\varepsilon)q}{\varepsilon}
	\Bigr)^{2q}.
\end{align*}
Thus there exist constants $C_3(\varepsilon),L_2(\varepsilon)>0$ such that
$$
\|a_{q,\varepsilon}^{(2q)}\|_{L^1(\mathbb R)}
\le
C_3(\varepsilon)\bigl(L_2(\varepsilon)q\bigr)^{2q},
$$
which proves \eqref{eq:a_derivative_final}.

Finally, we integrate by parts $2q$ times to obtain
$$
f_\varepsilon^{2q}(\tau)
=
c_n\tau^{-2q}(-1)^q
\int_{\mathbb R}a_{q,\varepsilon}^{(2q)}(t)\cos(t\tau)\,dt.
$$
Using \eqref{eq:a_derivative_final}, we conclude that there exists a constant $C_5(\varepsilon)>0$ such that
$$
|f_\varepsilon^{2q}(\tau)|
\le
c_n\tau^{-2q}
\|a_{q,\varepsilon}^{(2q)}\|_{L^1(\mathbb R)}
\le
C_5(\varepsilon)
\Bigl(\frac{L_2(\varepsilon)q}{\tau}\Bigr)^{2q}.
$$
This proves \eqref{eq:tail_decay_microlocal}.

	Fix $q\in\mathbb{Z}_+$ and set $l=2q+6$. By \eqref{eq:tail_decay_microlocal}, for $m\ge 2$, we have
	$$
	\sup_{\tau\in(m-1,m]}\bigl|f_\varepsilon^{2q+6}(\sqrt{\tau})\bigr|
	\le C(\varepsilon)\left(\frac{L(\varepsilon)q}{(m-1)^{1/2}}\right)^{2q+6}
	\le C(\varepsilon)\Bigl(\frac{2L(\varepsilon)q}{m^{1/2}}\Bigr)^{2q+6}.
	$$
	Substituting this into \eqref{eq:proj_reduction_microlocal}, we obtain
	\begin{equation*}
		\|\Delta_\bullet^q U_{n,\varepsilon}^{\bullet,\mathrm{tail},2q+6}\|^2_{L^2((s_0-\delta,s_0+\delta)\times U_x)}
		\le
		C(\varepsilon)(2L(\varepsilon)q)^{4q+12}
		\sum_{m=2}^{\infty} m^{-3}
		+ C_0.
	\end{equation*}
	Since the series $\sum_{m=2}^{\infty} m^{-3}$ converges, there exists $C_0(\varepsilon)>0$ such that
	$$
	\|\Delta_\bullet^q U_{n,\varepsilon}^{\bullet,\mathrm{tail},2q+6}\|_{L^2((s_0-\delta,s_0+\delta)\times U_x)}
	\le C_0(\varepsilon)(2L(\varepsilon)q)^{2q+6}.
	$$
	The proof of Lemma~\ref{lemma:KN_goal_final} is concluded by applying Stirling's formula \cite[p.~312, Eq.~(A.40)]{T_PDE_1}.
\end{proof}

	  	\bigskip

\noindent {\bf Acknowledgements.} The authors were supported in part by NSFC.  
Views and opinions expressed are those of the authors only and do not necessarily reflect those of the funding organizations.

\bigskip	\noindent {\bf Data Availability Statement.} Data sharing not applicable to this article as no datasets were generated or analysed during the current study.

\bigskip	\noindent {\bf Conflict of Interest.} The authors have no conflicts of interest to declare that are relevant to the content of this article.

%---------------------------------------------------------------------------------------%
\bibliography{ref}
\bibliographystyle{amsplain}

\end{document}